\documentclass[final,leqno,onefignum,onetabnum]{siamltex1213}

\usepackage{pifont}
\usepackage{amsmath}
\usepackage{amsfonts}
\usepackage{amscd}
\usepackage{latexsym}
\usepackage{amssymb}
\usepackage{stmaryrd}
\usepackage{overpic}
\usepackage{graphicx}
\usepackage{epstopdf}
\usepackage{subfig}
\newtheorem{remark}[theorem]{Remark}

\def\cP{\mathcal{P}}
\def\R{\mathbb R}
\renewcommand{\div}{{\rm div}}
\def\bsigma{\boldsymbol{\sigma}}
\def\cA{\mathcal {A}}
\def\mS{\mathbb S}
\def\mK{\mathbb K}
\def\btau{\boldsymbol{\tau}}
\def\cT{\mathcal{T}}
\def\cF{\mathcal{F}}
\def\bphi{\boldsymbol{\phi}}
\def\bzero{\boldsymbol{0}}
\def\mV{\mathbb V}
\def\mT{\mathbb T}
\def\mF{\mathbb F}
\def\btheta{\boldsymbol{\theta}}
\def\cO{\mathcal O}
\newcommand{\tr}{{\rm tr}}
\def\bI{\boldsymbol{I}}
\newcommand\re{{\mathrm{e}}}
\def\bepsilon{\boldsymbol{\epsilon}}

\newcommand{\vertiii}[1]{{\left\vert\kern-0.25ex\left\vert\kern-0.25ex\left\vert
#1 \right\vert\kern-0.25ex\right\vert\kern-0.25ex\right\vert}}

\title{Interior Penalty Mixed Finite Element Methods of Any Order in
Any Dimension for Linear Elasticity with Strongly Symmetric Stress
Tensor 
\thanks{
The work of the first and third authors was supported in part by the
DOE Grant DE-SC0009249 as part of the Collaboratory on Mathematics
for Mesoscopic Modeling of Materials and by DOE Grant DE-SC0014400
and NSF Grant DMS-1522615.
The work of the second and third authors was supported in part by
National Natural Science Foundation of China (NSFC) (Grant
No.91430215, 41390452) and by Beijing International Center for
Mathematical Research of Peking University, China. }
} 

\author{ 
Shuonan Wu\thanks{Department of Mathematics, The Pennsylvania
State University, University Park, PA, 16802, USA
(sxw58@psu.edu).} \and 
Shihua Gong\thanks{Beijing International Center for
Mathematical Research, Peking University, Beijing 100871, P. R.
China (gongshihua@pku.edu.cn).} \and  
Jinchao Xu\thanks{Corresponding author. Department of Mathematics, The
Pennsylvania State University, University Park, PA 16802, USA
(xu@math.psu.edu).}  
}

\begin{document}
\maketitle
\slugger{mms}{xxxx}{xx}{x}{x--x}%slugger should be set to mms, siap, sicomp, sicon, sidma, sima, simax, sinum, siopt, sisc, or sirev

\begin{abstract}
We propose two classes of mixed finite elements for linear elasticity
of any order, with interior penalty for nonconforming symmetric stress
approximation. One key point of our method is to introduce some
appropriate nonconforming face-bubble spaces based on the local
decomposition of discrete symmetric tensors, with which the stability
can be easily established. We prove the optimal error estimate for
both displacement and stress by adding an interior penalty term. The
elements are easy to be implemented thanks to the explicit
formulations of its basis functions. Moreover, the methods can be
applied to arbitrary simplicial grids for any spatial dimension in a
unified fashion. Numerical tests for both 2D and 3D are provided to
validate our theoretical results.
\end{abstract}



\pagestyle{myheadings}
\thispagestyle{plain}
\markboth{IP MIXED SYMMETRIC STRESS ELEMENTS FOR ELASTICITY}{S. WU, S.
GONG, AND J. Xu}

\section{Introduction} \label{sec:intro}

Mixed finite element methods for linear elasticity are popular methods
to approximate the stress-displacement system derived from
Hellinger-Reissner variational principle. However, it is more
difficult to develop the stable mixed finite element methods for
linear elasticity than that for scalar second-order elliptic problems,
as the stress tensor is required to be symmetric due to the
conservation of angular momentum. One way to circumvent this
difficulty is to use composite element techniques
\cite{johnson1978some, arnold1984family}.  Another approach is to use
some well-known $H(\div)$ elements to relax the symmetry. One main
technique is to introduce a Lagrange multiplier approximating the
non-symmetric part of the displacement gradient while enforcing stress
symmetry weakly \cite{amara1979equilibrium, arnold2007mixed,
boffi2009reduced, cockburn2010new, farhloul1997dual, qiu2009mixed,
gopalakrishnan2012second}.

The first stable non-composite finite element method for classical
mixed finite formulation of plane elasticity was proposed by Arnold
and Winther in 2002 \cite{arnold2002mixed}. In this class of
elements, the displacement is discretized by discontinuous piecewise
$\cP_k$ ($k\geq 1$) polynomial, while the stress is discretized by the
conforming $\cP_{k+2}$ tensors whose divergence is $\cP_k$ vector on
each triangle. The analogue of the results in 3D case were reported in
\cite{adams2005mixed, arnold2008finite}. All the results in this
series have some features in common: the degree of polynomial for the
displacement should satisfy $k\geq 1$. The similar idea can be applied
to the rectangular element, see \cite{arnold2005rectangular,
chen2011conforming, hu2014simple}.

Recently Hu and Zhang \cite{hu2014family, hu2015family} and Hu
\cite{hu2015finite} proposed a family of conforming mixed elements for
$\R^n$ that have fewer degrees of freedom than those in the earlier
literature.  For $k \geq n$, this class of elements are optimal in the
sense that the displacement is discretized by discontinuous piecewise
$\cP_k$ polynomial, while the stress is discretized by the conforming
$\cP_{k+1}$ tensors.  These elements also admit a unified theory and
a relatively easy implementation.  For the case that $k \leq n-1$, the
symmetric tensor spaces are enriched by proper high order $H(\div)$
bubble functions to stabilize the discretization \cite{hu2014finite}.
Similar mixed elements on rectangular and cuboid grids were
constructed in \cite{hu2015new}.

There have been also numerous works in the literature on nonconforming
mixed elements.  For rectangular or cuboid grids, we refer to
\cite{yi2005nonconforming, yi2006new, hu2007lower, awanou2009rotated,
man2009lower}.  For simplicial grids, we first refer Arnold and
Winther \cite{arnold2003nonconforming} (2D) and
\cite{arnold2014nonconforming} (3D).  These elements contain the
displacement space with $k=1$, but it is suboptimal as only the first
order accuracy can be proved for the displacement. In
\cite{gopalakrishnan2011symmetric}, Gopalakrishnan and Guzm\'{a}n
developed a family of simplicial elements for $k\geq 1$ in both two
and three dimensions. The optimal convergence order for the
displacement can be proved under the full elliptic regularity
assumption but the convergence order of $L^2$ error for stress is
still suboptimal.

All the aforementioned simplicial elements have the constraint that
$k\ge 1$.  For the lowest order case $k=0$ in 2D, Cai and Ye
\cite{cai2005mixed} used the Crouzeix-Raviart element to approximate
each component of the symmetric stress and piecewise constants for the
displacement. Their method was proved to be convergent by adding an
interior penalty term to weakly enforce the continuity of the stress.
As the authors claimed, their elements can be extended to higher
spatial dimensions, but it is not clear how the elements can be
extended to higher orders.

The purpose of this paper is to construct a family of mixed finite
elements ($k\geq 0$) for simplicial grids in any dimension. Precisely,
the piecewise $\cP_k$ vector space without interelement continuity is
applied to approximate the displacement.  To design the piecewise
$\cP_{k+1}$ spaces for the stress, the crucial point is to introduce
the conforming div-bubble spaces \cite{hu2015finite} and {\it
nonconforming face-bubble spaces}, with which the stability can easily
be established. We then add the spaces with two classes of spaces to
obtain the desired approximation property. The first class is locally
defined with elementwise degrees of freedom, while the second class
does not have local d.o.f.  but has a very small dimension. Any space
between these two classes can be proved to be convergent. Especially,
the finite element space proposed in \cite{cai2005mixed} in lowest
order lies in this case. Moreover, our first class of space is
precisely the space proposed in \cite{gopalakrishnan2011symmetric}
when $k\geq 1$, while the d.o.f are slightly different.

Due to the discontinuity of the normal stress on each interior face,
the stress-displacement mixed formulation is modified by adding an
interior penalty term to weakly enforce the continuity, which is a
standard technique for discontinuous Galerkin methods and also adopted
in \cite{cai2005mixed}. The convergence of our mixed finite element
method is studied according to the three ingredients step by step:
stability, approximation and consistency, with which a constructive
proof can be obtained naturally. More importantly, based on our
knowledge, our second class of spaces in lowest order has the smallest
dimension among all the mixed finite elements on simplicial grids
regardless whether the symmetry of stress is imposed strongly or
weakly.

The rest of the paper is organized as follows. In the next section, we
present the local decomposition of discrete symmetric tensors. In
section \ref{sec:fem}, we define two classes of finite element spaces
for symmetric tensors in any space dimension from the perspectives of
both stability and approximation property. In section
\ref{sec:IP-mixed}, the interior penalty mixed finite element method
is proposed, and its well-posedness and error analysis are given
subsequently. We then discuss the reduced elements in section
\ref{sec:reduced} and prove that the nonconforming elements have to be
applied in our framework when $k \leq n-1$. Numerical tests in both 2D
and 3D case will be given in section \ref{sec:numerical} and the
concluding remarks will then arrive to close the main text.

%% end of file %%
%%%%%%%%%%%%%%%%%%%%%%%%%%%%%%%%%%%%%%%%
%% Section 2 %% 
%%%%%%%%%%%%%%%%%%%%%%%%%%%%%%%%%%%%%%%%
\section{Local Decomposition of Discrete Symmetric Tensors}
\label{sec:preliminary}

%%%%%%%%%%%%%%%%%%%%%%%%%%%%%%%%%%%%%%%% 
%% Problem set-up, continuous  
%%%%%%%%%%%%%%%%%%%%%%%%%%%%%%%%%%%%%%%%
In this paper, we consider the following linear elasticity problem
with Dirichlet boundary condition  
\begin{equation} \label{equ:elasticity}
\left\{
\begin{aligned}
\cA \bsigma - \bepsilon(u) &= 0 \qquad \text{in~}\Omega, \\
\div \bsigma &= f \qquad \text{in~} \Omega, \\
u &= 0 \qquad \text{on~} \partial\Omega, 
\end{aligned}
\right.
\end{equation}
where $\Omega \subset \R^n$. The displacement and stress are denoted
by $u: \Omega \mapsto \R^n$ and $\bsigma: \Omega \mapsto \mS$,
respectively. Here, $\mS$ represents the space of real symmetric
matrices of order $n \times n$. The compliance tensor $\cA: \mS
\mapsto \mS$ is assumed to be bounded and symmetric positive definite.
The linearized strain tensor is denoted by $\bepsilon(u) = (\nabla u +
(\nabla u)^T)/2$.

The mixed formulation of \eqref{equ:elasticity} is to find $(\bsigma,
u) \in \Sigma \times V := H(\div, \Omega; \mS) \times L^2(\Omega;
\R^n)$, such that 
\begin{equation} \label{equ:mixed-formulation}
\left\{
\begin{aligned}
&(\cA \bsigma, \btau)_{\Omega} + (\div \btau, u)_{\Omega} &=&~ 0 \quad
&\forall \btau \in \Sigma, \\
&(\div \bsigma, v)_{\Omega} &=&~ (f, v)_{\Omega} \quad &\forall v \in
V.
\end{aligned}
\right.
\end{equation}
Here $H(\div, \Omega; \mS)$ consists of square-integrable symmetric
matrix fields with square-integrable divergence. The corresponding
$H(\div)$ norm is defined by 
\[
\|\btau\|_{\div, \Omega}^2 := \|\btau\|_{0, \Omega}^2 + \|\div
\btau\|_{0, \Omega}^2, \quad \forall \btau \in H(\div, \Omega; \mS).
\]
The $L^2(\Omega; \R^n)$ is the space of vector-valued functions which
are square-integrable with the standard $L^2$ norm.

Throughout this paper, we shall use letter $C$ to denote a generic
positive constant independent of $h$ which may stand for different
values at its different occurrences. The notation $x \lesssim y$ means
$x \leq Cy$ and $x \simeq y$ means $x \lesssim y \lesssim x$.

%%%%%%%%%%%%%%%%%%%%%%%%%%%%%%%%%%%%%%%% 
%% Preliminaries 
%%%%%%%%%%%%%%%%%%%%%%%%%%%%%%%%%%%%%%%% 
\subsection{Preliminaries}
Suppose that the domain $\Omega$ is subdivided by a family of shape
regular simplicial grids $\cT_h = \{K\}$. Let $h_K$ be the diameter of
element $K$, $h = \max_K h_K$ be the mesh diameter of $\cT_h$.
The set of all faces of $\cT_h$ is denoted by $\cF_h = \{F\}$ with
the diameter $h_F$ for face $F$. The set of faces can be divided into
two parts: the boundary faces set $\cF_h^{\partial} = \cF_h \cap
\partial\Omega$, and the interior faces set $\cF_h^i = \cF_h \setminus
\cF_h^{\partial}$. For any $F \in \cF_h$, the set of all elements
that share the face $F$ is denoted by $\cT_{h,F}$. The unit normal
vector with respect the face $F$ is represented by $\nu_F$.

Let $F \in \cF_h^i$ be the common face of two elements $K^+$ and
$K^-$, and $\nu_F^+$ and $\nu_F^-$ be the unit outward normal vectors
on $F$ with respect to $K^+$ and $K^-$, respectively. Then we define
the jump $[\cdot]$ on $F \in \cF_h^i$ for $\btau: \Omega \mapsto \mS$
by 
\begin{equation} \label{equ:jump}
[\btau] := \btau^+ \nu_F^+ + \btau^- \nu_F^-.
\end{equation}

For a given simplex $K$, its vertices are denoted by $a_1, \cdots,
a_{n+1}$. The face that does not contain the vertex $a_i$ is denoted
by $F_i$. The barycentric coordinates with respect to $K$ are
represented by $\lambda_1(x), \cdots, \lambda_{n+1}(x)$.  For any edge
$e_{ij} = a_j - a_i$ of element $K$, $i\neq j$, let $t_{ij}$ be the
unit tangent vectors along this edge, namely 
\[
t_{ij} := \frac{a_j - a_i}{|a_j - a_i|} = \frac{a_j -
  a_i}{|e_{ij}|}.
\]
Then we have the following important result describing the
relationship between the simplex $K$ and $\mS$. 

%%% Lemma: basis of S %%%
\begin{lemma} \label{lem:basis-S}
The symmetric tensors $\{t_{ij}t_{ij}^T, ~\forall i<j\}$ form a basis
of $\mS$.
\end{lemma}
\begin{proof}
See Lemma 2.1 in \cite{hu2015finite}.
\end{proof}

These symmetric matrices $t_{ij}t_{ij}^T$ of rank one are the basis
ingredients when constructing the finite elements for the symmetric
stress tensors. One of the most commonly used properties of these
basis functions is 
\begin{equation} \label{equ:basis-normal}
t_{ij} t_{ij}^T \nu_{F_s} = 0 \iff \forall s \neq i, j.
\end{equation}

When applying the standard $\cP_{k+1}$ Lagrangian element for each
entry of the symmetric tensors, we obtain the following
$\cP_{k+1}(\mS)$ Lagrangian element:
\begin{equation} \label{equ:lagrangian}
\Sigma_{k+1,h}^c := \{\btau\in H^1(\Omega; \mS) ~|~ \btau|_K \in
\cP_{k+1}(K; \mS)\}.
\end{equation}
Collect all the face-bubble functions in $\Sigma_{k+1,h}^c$, we have
the following $H^1(\mS)$ face-bubble function space 
\begin{equation} \label{equ:H1-face-bubble}
\Sigma_{k+1,h,f}^c := \left\{ \btau \in H^1(\Omega; \mS) ~|~ \btau|_K
\in \sum_{i=1}^{n+1} (\prod_{j=1,j\neq i}^{n+1} \lambda_j)
  \cP_{k+1-n}(K;\mS) \right\}.
\end{equation}
Here we define $\cP_{m}(K) = \{0\}$ if $m<0$. Clearly, the $H^1(\mS)$
face-bubble function space is nonempty only when $k\geq n-1$.

%%%%%%%%%%%%%%%%%%%%%%%%%%%%%%%%%%%%%%%% 
%% polynomial space in simplex 
%%%%%%%%%%%%%%%%%%%%%%%%%%%%%%%%%%%%%%%%
\subsection{Local decomposition of polynomial spaces}
In this subsection, we introduce some polynomial spaces and discuss
their relationships. 

\subsubsection{Some polynomial spaces in simplex $K$}
We first give the following lemma that simplifies the reader's
understanding.
\begin{lemma}\label{lem:polynomial-basis}
Suppose $\{\psi_0, \cdots, \psi_q\}(q\geq 0)$ are linearly
independent, and $\{\psi_l^k, k=0,1, \cdots \}$ are independent for
$1\leq l \leq q$. Then for any $k \geq 0$, 
$$ 
\left\{ \prod_{l=0}^q \psi_l^{m_l}, \sum_{l=0}^q m_l = k 
\right\} \text{~are linearly independent functions}.
$$ 
\end{lemma}
%% end of lemma %% 

Let $\lambda_0 = 1$. For a $n$-dimensional simplex $K$, it is
well-known that $\{\lambda_l, l=0, 1,\cdots, n+1\}$ is a set of
linearly dependent functions, which forms a basis of $\cP_1(K;\R)$ if
any one of them is removed. In light of Lemma
\ref{lem:polynomial-basis}, $\cP_k(K;\R)$ can be written as 
$$ 
\cP_k(K;\R) = \text{span} \left\{ 
\prod_{l=1}^{n+1} \lambda_l^{m_l}, \sum_{l=1}^{n+1} m_l = k 
\right\},
$$ 
or for $i\geq 1$, 
$$ 
\cP_k(K;\R) = \text{span} \left\{ \prod_{l=0, l \neq i}^{n+1}
\lambda_l^{m_l}, \sum_{l=0, l\neq i}^{n+1} m_l = k \right\} 
= \text{span} \left\{ \prod_{l=1, l \neq i}^{n+1} \lambda_l^{m_l},
\sum_{l=1, l\neq i}^{n+1} m_l \leq k \right\}.
$$ 

Now we introduce the spaces by removing two functions in $\{\lambda_l,
l=0,1, \cdots, n+1\}$. If $\lambda_0$ and $\lambda_i (1\leq i \leq n)$
are removed, we have the following space 
\begin{equation} \label{equ:hat-i}
\cP_k^{\hat{0},\hat{i}}(K;\R) := \text{span} \left\{ \prod_{l=1, l\neq
i}^{n+1} \lambda_l^{m_l}, \sum_{l=1, l\neq i}^{n+1} m_l = k
\right\}.
\end{equation}
If $\lambda_i, \lambda_j (1\leq i \neq j \leq n)$ are removed, we have 
\begin{equation} \label{equ:hat-ij}
\cP_k^{\hat{i}, \hat{j}}(K;\R) := \text{span} \left\{ \prod_{l=0,
l\neq i,j}^{n+1} \lambda_l^{m_l}, \sum_{l=0, l\neq i,j}^{n+1} m_l = k
\right\} = \text{span} \left\{ \prod_{l=1, l\neq i,j}^{n+1}
\lambda_l^{m_l}, \sum_{l=1, l\neq i,j}^{n+1} m_l \leq k \right\}.
\end{equation}

%%%%%%%%%%%%%%%%%%%%%%%%%%%%%%%%%%%%%%%%
%% restriction and extension %% 
%%%%%%%%%%%%%%%%%%%%%%%%%%%%%%%%%%%%%%%%
\subsubsection{Natural restriction and extension operators}
The restriction operator is defined as 
\begin{equation} \label{equ:restriction} 
\mathcal{R}_i: \cP_k(K;\R) \mapsto \cP_k(F_i;\R), \quad \mathcal{R}_i
p := p|_{F_i}, \quad \forall p\in \cP_k(K;\R).
\end{equation} 

For any $F_i \subset \partial K$, we have $\lambda_i|_{F_i} = 0$ and
for $l\neq i$, $\lambda_l^{F_i} = \lambda_l|_{F_i}$ are exactly the
barycentric coordinates on $F_i$.  For any $p\in \cP_k(F_i;\R)$, it
can be uniquely written under the basis $\{\lambda_l^{F_i}, l \neq
i,j\}$, i.e. 
$$
p = \sum_{|\boldsymbol{m}|=k} c_{\boldsymbol{m}} \prod_{l=0, l\neq
i,j}^{n+1} \lambda_{l}^{F_i, m_l}.
$$
Then the extension operator is denoted as  
\begin{equation} \label{equ:extension}
\begin{aligned}
\mathcal{E}_i^{\hat{j}}: \cP_k(F_i;\R) &\mapsto \cP_k(K;\R) \qquad 0
\leq j \neq i \leq n+1, \\
\mathcal{E}_i^{\hat{j}} p &:= \sum_{|\boldsymbol{m}|=k}
c_{\boldsymbol{m}} \prod_{l=0, l\neq i,j}^{n+1} \lambda_{l}^{m_l}.
\end{aligned}
\end{equation}
With the help of $\mathcal{R}_i$ and $\mathcal{E}_i^{\hat{j}}$, we
have the following properties:
\begin{lemma} \label{lem:spaces}
It holds that 
\begin{enumerate}
\item $\mathcal{R}_i \mathcal{E}_i^{\hat{j}} = id_{\cP_k(F_i;\R)},
  ~\forall j \neq i$.
\item $\ker(\mathcal{R}_i)\cap \cP_k(K;\R) = \lambda_i \cP_{k-1}(K;\R)$.
\item $\cP_k^{\hat{0},\hat{i}}(K;\R) =
  \mathrm{range}(\mathcal{E}_i^{\hat{0}}), \quad
  \cP_k^{\hat{i},\hat{j}}(K;\R) =
  \mathrm{range}(\mathcal{E}_i^{\hat{j}}) =
  \mathrm{range}(\mathcal{E}_j^{\hat{i}})$.
%%\item $\cP_k(K;\R) = \cP_k^{\hat{0},\hat{i}}(K;\R) \oplus \lambda_i
%%  \cP_{k-1}(K;\R), ~\forall 1\leq i \leq n+1$.
\item $\cP_k^{\hat{0},\hat{i}}(K;\R) \cong \cP_k(F_i;\R), ~\forall 1
  \leq i \leq n+1$.
\item $\cP_k^{\hat{i},\hat{j}}(K;\R)|_{F_i} \cong \cP_k(F_i;\R),
  ~\cP_k^{\hat{i},\hat{j}}(K;\R)|_{F_j} \cong \cP_k(F_j;\R), ~\forall
  1 \leq i < j \leq n+1$.
\item $\lambda_i \cP_k^{\hat{0},\hat{j}}(K;\R) \cap \lambda_j
  \cP_k^{\hat{0},\hat{i}}(K;\R) = \{0\}, ~\forall 1\leq i \neq j \leq
  n+1$.
\end{enumerate}
\end{lemma}
%% proof of the properties %%
\begin{proof}
The properties 1-5 are derived from the definition of natural
restriction and extension operators.  For any $p\in \lambda_i
\cP_k^{\hat{0},\hat{j}}(K;\R) \cap \lambda_j
\cP_k^{\hat{0},\hat{i}}(K;\R)$, we immediately have $\mathcal{R}_i p =
0$ and $p = \lambda_j \mathcal{E}_i^{\hat{0}}q$, where $q \in
\cP_k(F_i;\R)$. Then 
\[
0 = \mathcal{R}_i p = \lambda_j^{F_i} \mathcal{R}_i
\mathcal{E}_i^{\hat{0}} q = \lambda_j^{F_i}q,
\]
which implies $q = 0$ thus $p = 0$.
\end{proof}

Without loss of clarity in what follows, we will use same notation
$\lambda_l$ for barycentric coordinates of both $K$ and $F$.

%%%%%%%%%%%%%%%%%%%%%%%%%%%%%%%%%%%%%%%% 
%% local decomposition of P_{k+1}
%%%%%%%%%%%%%%%%%%%%%%%%%%%%%%%%%%%%%%%% 
\subsubsection{Local Decomposition of $\cP_{k+1}(K;\R)$}
We first give the following lemma.
\begin{lemma} \label{lem:direct-decomposition}
Let $\tilde{\mathcal{R}}:V\mapsto W$ and $\tilde{\mathcal{E}}:W\mapsto
V$ be bounded linear operators between Banach spaces. If
$\tilde{\mathcal{X}} = \tilde{\mathcal{R}}\tilde{\mathcal{E}}$ is an
isomorphism on $W$,
then 
\begin{equation} \label{equ:direct-decomposition}
V = \ker(\tilde{\mathcal{R}}) \oplus
\mathrm{range}(\mathcal{\tilde{E}}).
\end{equation}
\end{lemma}
%% begin proof %% 
%%\begin{proof}
%%For any $v \in \ker(\mathcal{\tilde{R}}) \cap
%%\text{range}(\tilde{\mathcal{E}})$, there exists $w \in W$, such that
%%$v = \tilde{\mathcal{E}}w$. Therefore, 
%%$$ 
%%0 = \tilde{\mathcal{R}}v = \tilde{\mathcal{R}}\tilde{\mathcal{E}}w =
%%\tilde{\mathcal{X}}w,
%%$$ 
%%which yields $w = 0$ since $\tilde{\mathcal{X}}$ is isomorphism. This
%%implies $\ker(\tilde{\mathcal{R}}) \cap
%%\text{range}(\tilde{\mathcal{E}}) = \{0\}$. On the other way, $\forall
%%v\in V$, let $v_E =
%%\tilde{\mathcal{E}}\tilde{\mathcal{X}}^{-1}\tilde{\mathcal{R}}v \in
%%\text{range}(\tilde{\mathcal{E}})$. Then 
%%$$ 
%%\tilde{\mathcal{R}}(v-v_E) = \tilde{\mathcal{R}}v -
%%\tilde{\mathcal{R}} \tilde{\mathcal{E}}\tilde{\mathcal{X}}^{-1}
%%\tilde{\mathcal{R}} v = 0,
%%$$ 
%%which means $v - v_E \in \ker(\tilde{\mathcal{R}})$.
%%\end{proof}
\begin{remark}
Take $\tilde{\mathcal{R}} = \mathcal{R}_i$ and $\tilde{\mathcal{E}} =
\mathcal{E}_i^{\hat{j}}$ in Lemma \ref{lem:direct-decomposition}, we
immediately have 
\begin{equation} \label{equ:decomposition-remark}
\cP_k(K;\R) = \lambda_i \cP_{k-1}(K;\R) \oplus
\cP_k^{\hat{i},\hat{j}}(K;\R) \qquad 0\leq j \neq i \leq n+1.
\end{equation}
\end{remark}
%% end proof %% 

Let $\cP_k^\perp(F_j;\R) \subset \cP_{k+1}(F_j;\R)$ be the $L^2$
orthogonal complement $\cP_k(F_j;\R)$ in $\cP_{k+1}(F_j;\R)$, namely
for $1\leq j \leq n+1$, 
$$ 
\begin{aligned}
\cP_{k+1}(F_j;\R) &= \Pi_{k,F_j}^0 \cP_{k+1}(F_j;\R) \oplus
(I-\Pi_{k,F_j}^0)\cP_{k+1}(F_j;\R)  \\ 
&= \cP_k(F_j;\R) \oplus \cP_k^\perp(F_j;\R), 
\end{aligned}
$$ 
where $\Pi_{k,F_j}^0$ is the $L^2$ projection operator to
$\cP_k(F_j;\R)$. Now we present the local decomposition of
$\cP_{k+1}(K;\R)$ as follows.

%% scalar decomposition %% 
\begin{theorem}\label{thm:scalar-decomposition}
For any given $1\leq i < j \leq n+1$, it holds that 
\begin{equation} \label{equ:scalar-decomposition}
\begin{aligned}
\cP_{k+1}(K;\R) = \lambda_i \lambda_j \cP_{k-1}(K;\R) &\oplus \lambda_j
\mathcal{E}_i^{\hat{0}}\cP_k(F_i;\R) \\ 
&\oplus \lambda_i \mathcal{E}_j^{\hat{0}}\cP_k(F_j;\R) \oplus
\mathcal{E}_j^{\hat{i}}\cP_k^\perp(F_j;\R).
\end{aligned}
\end{equation}
\end{theorem}
%% proof %% 
\begin{proof}
Take $\tilde{\mathcal{R}} = \mathcal{R}_j$, $\tilde{\mathcal{E}} =
\lambda_i \mathcal{E}_j^{\hat{0}}\Pi_{k,F_j}^0 +
\mathcal{E}_j^{\hat{i}}(I-\Pi_{k,F_j}^0)$ in Lemma
\ref{lem:direct-decomposition}. A simple calculation shows that
$\tilde{\mathcal{X}} = \tilde{\mathcal{R}}\tilde{\mathcal{E}}:
\cP_{k+1}(F_j;\R)\mapsto \cP_{k+1}(F_j;\R)$. Since
$\dim(\cP_{k+1}(F_j;\R)) < \infty$, we only need to check that 
$\tilde{\mathcal{X}}$ is one-to-one to prove it isomorphism. For any
$p_j \in \ker(\tilde{\mathcal{X}})$, we have 
$$ 
0 = \tilde{\mathcal{X}}p_j
= \mathcal{R}_j\lambda_i \mathcal{E}_j^{\hat{0}} \Pi_{k,F_j}^0 p_j 
+ \mathcal{R}_j\mathcal{E}_j^{\hat{i}} (I-\Pi_{k,F_j}^0)p_j
= \lambda_i \Pi_{k,F_j}^0 p_j + (I-\Pi_{k,F_j}^0)p_j.
$$ 
Apply $\Pi_{k,F_j}^0$ on both sides, we have 
$$ 
\Pi_{k,F_j}^0 \left(\lambda_i \Pi_{k,F_j}^0 p_j \right) = 0 \quad
\text{or}\quad \int_{F_j} \lambda_i \Pi_{k,F_j}^0 p_j q = 0,~\forall
q\in \cP_k(F_j;\R),
$$ 
which implies $\Pi_{k,F_j}^0 p_j = 0$ by taking $q = \Pi_{k,F_j}p_j$.
Then $(I-\Pi_{k,F_j}^0)p_j = 0$ thus $p_j = 0$.

In light of Lemma \ref{lem:direct-decomposition}, we have 
\begin{equation} \label{equ:decom-RE} 
\cP_{k+1}(K;\R) = \ker(\mathcal{R}_j) \oplus
\mathrm{range}(\mathcal{\tilde{E}}).
\end{equation}
From Lemma \ref{lem:spaces} and \eqref{equ:decomposition-remark},
\begin{equation} \label{equ:decom-R} 
\begin{aligned}
\ker(\mathcal{R}_j)\cap \cP_{k+1}(K;\R) = \lambda_j \cP_k(K;\R) &=
\lambda_j \left(
    \lambda_i\cP_{k-1}(K;\R) \oplus
    \mathcal{E}_i^{\hat{0}}\cP_k(F_i;\R) \right) \\ 
   &= \lambda_i\lambda_j \cP_{k-1}(K;\R) \oplus \lambda_j
   \mathcal{E}_i^{\hat{0}}\cP_k(F_i;\R).
\end{aligned}
\end{equation} 
And 
\begin{equation} \label{equ:decom-E} 
\mathrm{range}(\tilde{\mathcal{E}}) = \lambda_i
\mathcal{E}_j^{\hat{0}}\cP_k(F_j;\R) +
\mathcal{E}_j^{\hat{i}}\cP_k^\perp(F_j;\R). 
\end{equation}
If $p \in \lambda_i \mathcal{E}_j^{\hat{0}}\cP_k(F_j;\R) \cap
\mathcal{E}_j^{\hat{i}}\cP_k^\perp(F_j;\R)$, then $\mathcal{R}_j p \in 
\lambda_i \cP_k(F_j;\R) \cap \cP_k^\perp(F_j;\R) = \{0\}$, which
implies $p \in \ker(\mathcal{R}_j)$. Then we have $p = 0$ in light of
\eqref{equ:decom-RE}, which means the sum in \eqref{equ:decom-E} is
direct. Take \eqref{equ:decom-R} and \eqref{equ:decom-E} into
\eqref{equ:decom-RE}, we obtain the local decomposition
\eqref{equ:scalar-decomposition}.
\end{proof}
%% end proof %% 

For the last term in \eqref{equ:scalar-decomposition}, we will show
its symmetry with respect to $i$ and $j$.
%% symmetric i,j %%
\begin{lemma}
It holds that  
\begin{equation} \label{equ:symmetric-ij}
\mathcal{E}_j^{\hat{i}}\cP_k^\perp(F_j;\R) =
\mathcal{E}_i^{\hat{j}}\cP_k^\perp(F_i;\R).
\end{equation}
\end{lemma}

\begin{proof}
Note that $\mathcal{E}_j^{\hat{i}}\cP_k(F_j;\R) =
\mathcal{E}_i^{\hat{j}}\cP_k(F_i;\R)$, then for any $p \in
\mathcal{E}_i^{\hat{j}}\cP_k^\perp(F_i;\R)$ and $q_j \in
\cP_k(F_j;\R)$, there exists $p_j \in \cP_{k+1}(F_j;\R)$ and $q_i\in
\cP_k(F_i;\R)$, such that 
$$ 
p = \mathcal{E}_j^{\hat{i}}p_j, \qquad q_i = \mathcal{R}_i
\mathcal{E}_j^{\hat{i}}q_j.
$$ 
Hence, $p \in \mathcal{E}_i^{\hat{j}}\cP_k^\perp(F_i;\R)$ implies that 
\begin{equation} \label{equ:perp-i} 
\int_{F_i} \mathcal{R}_i \mathcal{E}_j^{\hat{i}}p_j \cdot
\mathcal{R}_i \mathcal{E}_j^{\hat{i}} q_j \mathrm{d} x = 0.
\end{equation} 
Define the affine mapping $A^{i,j}: F_i \mapsto F_j$ by 
$$ 
A^{i,j}(a_s) = a_s, s \neq i,j \quad \text{and} \quad A^{i,j}(a_i) = a_j.
$$ 
It is straightforward that 
$$ 
\lambda_s^{F_i}(x) = \lambda_s^{F_j}(A^{i,j}(x)), s \neq i,j \quad
\text{and} \quad \lambda_j^{F_i}(x) = \lambda_i^{F_j}(A^{i,j}(x)), 
$$ 
and 
$$ 
(\mathcal{R}_i \mathcal{E}_j^{\hat{i}}f_j)(x) = f_j(A^{i,j}(x)) \qquad
\forall f_j \in \cP_k(F_j;\R).
$$ 
Then \eqref{equ:perp-i} implies 
$$ 
0 = \int_{F_i} p_j(A^{i,j}(x))\cdot q_j(A^{i,j}(x))\mathrm{d}x =
\det(DA^{i,j})^{-1} \int_{F_j} p_j(y) q_j(y) \mathrm{d}y,
$$ 
where $DA^{i,j}$ is the Jaboci of $A^{i,j}$. Then $p_j \in
\cP_k^\perp(F_j;\R)$ thus $p \in
\mathcal{E}_j^{\hat{i}}\cP_k^\perp(F_j;\R)$. Therefore,
$\mathcal{E}_i^{\hat{j}}\cP_k^\perp(F_j;\R) \subset
\mathcal{E}_j^{\hat{i}}\cP_k^\perp(F_i;\R)$.
\end{proof}

%%%%%%%%%%%%%%%%%%%%%%%%%%%%%%%%%%%%%%%% 
%% Local decomposition of \cP(\mS)
%%%%%%%%%%%%%%%%%%%%%%%%%%%%%%%%%%%%%%%% 
\subsection{Local decomposition of $\cP_{k+1}(K;\mS)$}

In light of Theorem \ref{thm:scalar-decomposition} and Lemma
\ref{lem:basis-S}, we immediately have the local decomposition of
$\cP_{k+1}(K;\mS)$ as 
\begin{equation} \label{equ:local-decomposition1}
\begin{aligned}
\cP_{k+1}(K;\mS) = \bigoplus_{1\leq i < j\leq n+1} \Big( 
    \lambda_i\lambda_j \cP_{k-1}(K;\R) &\oplus \lambda_j
    \cP_k^{\hat{0},\hat{i}}(K;\R)  \\
    &\oplus \lambda_i \cP_k^{\hat{0},\hat{j}}(K;\R) \oplus
    \mathcal{E}_j^{\hat{i}}\cP_k^\perp(F_j;\R)
    \Big) t_{ij}t_{ij}^T.
\end{aligned}
\end{equation}
Therefore, we can define the following three spaces: 
\begin{enumerate}
%% local conforming div-bubble %% 
\item {\em local conforming div-bubble function spaces} (see also
\cite{hu2015finite}) 
\begin{equation} \label{equ:local-c-div-bubble}
\Sigma_{k+1,h,b}(K) := \bigoplus_{1\leq i<j\leq n+1} \lambda_i
\lambda_j \cP_{k-1}(K;\R)t_{ij}t_{ij}^T.
\end{equation}
%% face-bubble %%
\item {\em local face-bubble function spaces} 
\begin{equation} \label{equ:local-face-bubble}
\begin{aligned}
\tilde{\Sigma}_{k+1,h,f}(K) &:= \bigoplus_{1\leq i < j \leq n+1}
\left( \lambda_i \cP_k^{\hat{0}, \hat{j}}(K;\R) \oplus \lambda_j
    \cP_k^{\hat{0},\hat{i}}(K;\R) \right) t_{ij}t_{ij}^T, \\
    &:= \bigoplus_{i=1}^{n+1} \tilde{\Sigma}_{k+1,h,F_i}(K),
\end{aligned}
\end{equation}
where 
\begin{equation} \label{equ:1-face-bubble}
\tilde{\Sigma}_{k+1,h,F_i}(K) := \bigoplus_{j=1, j\neq
i}^{n+1}\lambda_j \cP_k^{\hat{0},\hat{i}}(K;\R)t_{ij}t_{ij}^T.
\end{equation}
%% non-conforming div-bubble
\item {\em local nonconforming div-bubble function spaces}
\begin{equation} \label{equ:local-nc-div-bubble}
\tilde{\Sigma}_{k+1,h,b}(K) :=\bigoplus_{1\leq i < j \leq n+1}
\mathcal{E}_j^{\hat{i}}\cP_k^\perp(F_j;\R)t_{ij}t_{ij}^T.
\end{equation}
\end{enumerate}

The following local decomposition of $\cP_{k+1}(K;\mS)$ then follows
from the definition of spaces and \eqref{equ:local-decomposition1}
directly.
%%%%%%%%%%%%%%%%%%%%%%%%%%%%%%%%%%%%%%%%%%
%%%% Local decomposition %%
%%%%%%%%%%%%%%%%%%%%%%%%%%%%%%%%%%%%%%%%%%
\begin{theorem} \label{thm:local-decomposition}
It holds that 
\begin{equation} \label{equ:local-decomposition}
\cP_{k+1}(K;\mS) = \Sigma_{k+1,h,b}(K) \oplus
\tilde{\Sigma}_{k+1,h,f}(K) \oplus \tilde{\Sigma}_{k+1,h,b}(K).
\end{equation}
\end{theorem}

%%%%%%%%%%%%%%%%%%%%%%%%%%%%%%%%%%%%%%%% 
%% d.o.f local face-bubble
%%%%%%%%%%%%%%%%%%%%%%%%%%%%%%%%%%%%%%%%
\subsection{Unisolvent set of degrees of freedom for local face-bubble
function spaces}
From \eqref{equ:1-face-bubble} and Lemma \ref{lem:spaces}, we have 

\[
\begin{aligned}
\tilde{\Sigma}_{k+1,h,F_i}(K)\nu_{F_i}|_{F_i} &= \sum_{j=1,j\neq i}^{n+1}
\lambda_j \mathcal{R}_i \left(\cP_k^{\hat{0},\hat{i}}(K;\R)\right)
t_{ij} (t_{ij}^T\nu_{F_i}) \\
  &= \sum_{j=1,j\neq i}^{n+1} \lambda_j \cP_k(F_i;\R) t_{ij} \\
  &= T^{\hat{i}}D_\lambda^{\hat{i}} \cP_k(F_i;\R^n),
\end{aligned}
\]
%%\[
%%\begin{aligned}
%%\tilde{\Sigma}_{k+1,h,F}(K)\nu_{F}|_{F} &= \sum_{x_j \in F}^{n+1}
%%\lambda_j \mathcal{R}_i \left(\cP_k^{\hat{0},\hat{i}}(K;\R)\right)
%%t_{ij} (t_{ij}^T\nu_{F_i}) \\
%%  &= \sum_{x_j \in F}^{n+1} \lambda_j \cP_k(F;\R) t_{ij} \\
%%  &= T^{\hat{i}}D_\lambda^{\hat{i}} \cP_k(F;\R^n),
%%\end{aligned}
%%\]
where $D_\lambda^{\hat{i}} = \diag(\lambda_1, \cdots, \lambda_{i-1},
\lambda_{i+1}, \cdots, \lambda_{n+1}), T^{\hat{i}} =
(t_{i1},\cdots, t_{i,i-1}, t_{i,i+1}, \cdots, t_{i,n+1}) \in
\R^{n\times n}$. It is apparent that $\det(T^{\hat{i}}) \neq 0$,
and one inner product of $\cP_k(F_i;\R^n)$ can be defined as  
\begin{equation} \label{equ:inner-product-F}
\langle \cdot, \cdot \rangle_{D_\lambda^{\hat{i}}} := \int_{F_i}
D_\lambda^{\hat{i}}p\cdot q \qquad \forall p, q\in \cP_k(F_i;\R^n).
\end{equation}
Therefore, the unisolvent set of d.o.f. for
$\tilde{\Sigma}_{k+1,h,F_i}(K)$ can be written as 
\begin{equation} \label{equ:dof-face1}
N_{F_i}^{\mu}(\btau) := \int_{F_i} \btau\nu_{F_i} \cdot \mu \qquad
\forall \mu\in\cP_k(F_i;\R^n).
\end{equation}

\paragraph{Basic functions for a specific set of degrees of freedom}
Denote $\{\varphi_{F_i,t}, t=1, \cdots, C_{k+n-1}^k\}$ as a basis of
$\cP_k(F_i; \R)$. For convenience, $\varphi_{F_i,t}$ are normalized
such that $\frac{1}{|F_i|}\int_{F_i} \varphi_{F_i,t}^2 = 1$. Then 
$$ 
\cP_k(F_i;\R^n) = \text{span}\left\{ 
\varphi_{F_i,t}e_m, ~t=1,\cdots, C_{k+n-1}^k, m=1,\cdots, n
\right\},
$$ 
where $e_m~(m=1,\cdots, n)$ are the unit vectors in $\R^n$. Hence, the
set of d.o.f. defined in \eqref{equ:dof-face1} is equivalent to 
\begin{equation} \label{equ:dof-face}
N_{F_i}^{t,m}(\btau) := \int_{F_i} \btau \nu_{F_i} \varphi_{F_i,t}
\cdot e_m \qquad t = 1, \cdots, C_{k+n-1}^k,~ m=1, \cdots, n.
\end{equation}

%% Theorem: basis function face-bubble %%  
\begin{theorem} \label{thm:face-bubble}
The basis functions for $N_{F_i}^{t,m}(\cdot)$ can be written as
\begin{equation} \label{equ:face-bubble-expression}
\bphi_{F_i}^{s,l} := \frac{1}{|F_i|} \sum_{1 \leq j \leq n+1, j\neq i}
\frac{\alpha_{ij}^l}{t_{ij} \cdot \nu_{F_i}} \lambda_j
\varphi_{j,s}^{\hat{0},\hat{i}} t_{ij}t_{ij}^T, 
\end{equation}
where $ \sum_{1\leq j\leq n+1, j\neq i} \alpha_{ij}^l t_{ij} = e_l$,
and $\varphi_{j,s}^{\hat{0},\hat{i}} \in
\cP_k^{\hat{0},\hat{i}}(K;\R)$ are uniquely determined by  
\begin{equation} \label{equ:face-projection}
\langle \varphi^{\hat{0},\hat{i}}_{j,s}, \varphi_{F_i,t}
\rangle_{\lambda_j} := \frac{1}{|F_i|}\int_{F_i} \lambda_j
\hat{\varphi}_{j,s}^{\hat{0},\hat{i}} \varphi_{F_i,t} = \delta_{st}
\qquad t = 1, \cdots, C_{n-1+k}^k,
\end{equation}
\end{theorem}

%% proof %% 
\begin{proof}
The lemma is followed by 
\[
\begin{aligned}
N_{F_i}^{t,m}(\bphi_{F_i}^{s,l}) &= \left( \int_{F_i}
\bphi_{F_i}^{s,l}\nu_{F_i} \varphi_{F_i,t} \right) \cdot e_m \\
&= \frac{1}{|F_i|} \left( \sum_{1\leq j \leq n+1, j\neq i}
  \int_{F_i} \alpha_{ij}^l\lambda_j \varphi_{j,s}^{\hat{0},\hat{i}}
  \varphi_{F_i,t} t_{ij} \right) \cdot e_m \\
&= \delta_{st} \left( \sum_{1 \leq j \leq n+1, j\neq i} \alpha_{ij}^l
    t_{ij} \right) \cdot e_m = \delta_{st} \delta_{lm}.
\end{aligned}
\]
\end{proof}

We can have the explicit formulation of the coefficient
$\alpha_{ij}^l$ in \eqref{equ:face-bubble-expression} as follows.
\begin{lemma} \label{lem:decomposition-explicit}
Given $i$, for any vector $v$, we have 
\begin{equation} \label{equ:alpha-explicit}
v = \sum_{1\leq j \leq n+1, j\neq i} v\cdot (\nabla \lambda_j)
  |e_{ij}| t_{ij}.
\end{equation}
\end{lemma}
\begin{proof}
For $u_h \in \mathcal{P}_1(K)$, we write $u_h = \sum_{i=j}^{n+1} u_j
\lambda_j$. Let $\xi = \nabla u_h \in \mathbb{R}^n$. Then, 
$$ 
\begin{aligned}
|K| v \cdot \xi &= (v, \nabla u_h)_K = \sum_{j=1}^{n+1} (v, \nabla
    \lambda_j)_K u_j = \sum_{1\leq j \leq n+1, j\neq i} (v, \nabla
      \lambda_j)_K (u_j - u_i) \\ 
    &= \sum_{1 \leq j \leq n+1, j \neq i} (v, \nabla \lambda_j)_K
    |e_{ij}| t_{ij} \cdot \xi,
\end{aligned}
$$ 
which implies \eqref{equ:alpha-explicit}.
\end{proof}

In light of Lemma \ref{lem:decomposition-explicit}, we have 
$$ 
\alpha_{ij}^l = e_l  \cdot (\nabla \lambda_j) |e_{ij}| \quad
\text{and} \quad  
\bphi_{F_i}^{s,l} = \frac{1}{|F_i|} \sum_{1 \leq j \leq n+1, j\neq i}
\frac{e_l \cdot (\nabla \lambda_j) |e_{ij}|}{t_{ij} \cdot \nu_{F_i}}
\lambda_j \varphi_{j,s}^{\hat{0},\hat{i}} t_{ij}t_{ij}^T. 
$$

%% Remark: lowest order %%
\begin{remark}
For the lowest case $k=0$, we immediately obtain that $\varphi_{F_i,1}
= 1$ and $\varphi_{j,1}^{\hat{0},\hat{i}} = n, \forall i,j$ by
\eqref{equ:face-projection}. Therefore, basis functions
\eqref{equ:face-bubble-expression} have the following formulation 
\begin{equation} \label{equ:face-bubble0}
\bphi_{F_i}^{1,l} = \frac{1}{|F_i|} \sum_{1 \leq j \leq n+1, j\neq i}
\frac{n e_l \cdot (\nabla \lambda_j)|e_{ij}|
}{t_{ij}\cdot\nu_{F_i}} \lambda_j t_{ij}t_{ij}^T.
\end{equation}
\end{remark}

In light of the formulation of $\bphi_{F_i}^{s,l}$ in
\eqref{equ:face-bubble-expression}, we have the following properties
of the face-bubble $\bphi_{F_i}^{s,l}$ by standard scaling argument. 
%% Lemma: scaling argument %%
\begin{lemma} \label{lem:scaling}
For any $K \in \cT_h$ and $F_i \subset \partial K$, we have  
\begin{subequations}
\begin{align}
\|\bphi_{F_i}^{s,l}\|_{0,K} &\lesssim h_K^{-n/2+1},   \label{equ:scaling-K0}
\\
\|\bphi_{F_i}^{s,l}\|_{\div, K} &\lesssim h_K^{-n/2}, \label{equ:scaling-K}\\
\|\bphi_{F_i}^{s,l} \nu_{F_i}\|_{0,F_i} &\lesssim h_K^{-(n-1)/2}. \label{equ:scaling-F}
\end{align}
\end{subequations}
\end{lemma}

%%%%%%%%%%%%%%%%%%%%%%%%%%%%%%%%%%%%%%%% 
%% setion 3 %% 
%%%%%%%%%%%%%%%%%%%%%%%%%%%%%%%%%%%%%%%%
\section{Stability and Approximation Property} 
\label{sec:fem}
%% displacement spaces %%
For the discretization of displacement, the most natural space is the
full $C^{-1}-\cP_k$ space 
\begin{equation} \label{equ:u-space}
V_{k,h} := \{v\in L^2(\Omega;\R^n)~|~v|_K \in \cP_k(K; \R^n)\}.
\end{equation}
For the discretization of symmetric stress, we try to find some
good approximation spaces under the constrain that the degree of
polynomials are at most $k+1$. To this end, we will discuss the
effects of different components in the local decomposition
\eqref{equ:local-decomposition}.

%%%%%%%%%%%%%%%%%%%%%%%%%%%%%%%%%%%%%%%% 
%% conforming element-bubble 
%%%%%%%%%%%%%%%%%%%%%%%%%%%%%%%%%%%%%%%%
\subsection{Stability for $R_k^{\perp}$: conforming div-bubble
function spaces}

Combine the local conforming $\div$-bubble functions in
\eqref{equ:local-c-div-bubble} element by element, we obtain the
{\em conforming $\div$-bubble function spaces} 
\begin{equation} \label{equ:conforming-bubble}
\Sigma_{k+1,h,b} := \{\btau ~|~ \btau|_K \in \Sigma_{k+1,h,b}(K),~\forall
K\in \cT_h\},
\end{equation}
which satisfies the $\btau \nu_F = 0$ for any $F\in \cF_h$. Hu
\cite{hu2015finite} also proved that $\Sigma_{k+1,h,b}$ are exactly
the full $H(\div;\mS)$ bubble function spaces. We note that the
conforming $\div$-bubble spaces are non-trivial when the degrees of
stress tensor spaces are quadratic at least ($k+1 \geq 2$).
$\Sigma_{k+1,h,b}$ was introduced in \cite{hu2015finite} to control
the orthogonal complement of the rigid motion space. Precisely, let 
\begin{equation} \label{equ:rigid-motion}
\begin{aligned}
R_k(K) &:= \{v\in \cP_k(K;\R^n)~|~(\nabla v + \nabla v^T)/2 = 0\}, \\
R_k &:= \{v\in V_{k,h}~|~ v|_K \in R_k(K),~\forall K \in \cT_h\},
\end{aligned}
\end{equation}
and 
\begin{equation} \label{equ:rigid-perp}
\begin{aligned}
R^{\perp}_k(K) &:= \{v\in \cP_k(K;\R^n)~|~(v,w)_K = 0 \text{~for any~}
w\in R(K)\}, \\
R^{\perp}_k &:= \{v \in V_{k,h} ~|~ v|_K \in R^{\perp}_k(K),~\forall K
\in \cT_h\}.
\end{aligned}
\end{equation}
It is easy to check that $R_0 = V_{0,h}$, namely the rigid motion
space in lowest order is piecewise constant vector space. Together
with the higher order case given by Hu \cite{hu2015finite}, we have
the following lemma.  

%%% Lemma: conforming bubble %%%
\begin{lemma} \label{lem:conforming-bubble}
It holds that 
\begin{equation} \label{equ:div-rigid-perp}
\div \Sigma_{k+1,h,b} = R^{\perp}_k \qquad \forall k \geq 0.
\end{equation}
\end{lemma}

\begin{proof}
The proof is presented here for the completeness. First,
\eqref{equ:div-rigid-perp} is trivially true for $k=0$. Now, we assume
$k\geq 1$. The definition of $R_k$ implies $\div \Sigma_{k+1,h,b}
\subset R_k^{\perp}$. Next we prove that only the zero function $v\in
R_k^{\perp}(K)$ satisfies 
\begin{equation} \label{equ:div-conforming-bubble}
\int_{K} \div \btau \cdot v = -\int_{K} \btau : \bepsilon(v) =  0
\qquad \forall \btau\in \Sigma_{k+1,h,b}(K).
\end{equation}
By Lemma \ref{lem:basis-S}, there exists a basis of $\mS$ dual to
$\{t_{ij}t_{ij}^T, 1 \leq i < j \leq n+1\}$ under the inner product
$\langle A, B \rangle := A:B$, denoted as $\{M_{ij}, 1\leq i<j\leq
n+1\}$. Notice that $\bepsilon(v) \in \cP_{k-1}(K;\mS)$, let 
$$
\bepsilon(v) = \sum_{1\leq i < j \leq n+1} q_{ij}M_{ij} \qquad
q_{ij}\in \cP_{k-1}(\R).
$$ 
Take $\btau = \sum_{1\leq i<j \leq n+1} \lambda_i\lambda_j q_{ij}
t_{ij}t_{ij}^T$ in \eqref{equ:div-conforming-bubble} to have 
\[
0 = \sum_{1\leq i < j \leq n+1} \int_K \lambda_i \lambda_j q_{ij}^2, 
\]
which implies $q_{ij} = 0$, thus $v\in R_k(K) \cap R_k^{\perp}(K) = 0$. 
\end{proof}
%%% end of proof %%%

It follows from the definition of $R_k$ and $R_k^{\perp}$ that
$V_{k,h} = R_k \oplus R_k^{\perp}$. Therefore, for any given $v_h \in
V_{k,h}$, there uniquely exist $v_{h,R} \in R_k$ and $v_{h,R^{\perp}}
\in R_k^{\perp}$ such that $v_h = v_{h,R} + v_{h,R^{\perp}}$. By $L^2$
orthogonality,  
$$ 
\|v_{h,R}\|_{0}^2 + \|v_{h,R^{\perp}}\|_0^2 = \|v_h\|_0^2.
$$ 
When constructing the stable pair $\Sigma_{k+1,h}-V_{k,h}$ of mixed
finite elements for elasticity, one key step is to find the stable
$\btau_h \in \Sigma_{k+1,h}$ that $\div \btau_h = v_h$. The following
lemma implies that the conforming $\div$-bubble spaces solve the
orthogonal complement of the rigid motion. 

%%% Lemma: stable div --> perp %% 
\begin{lemma} \label{lem:v-perp}
For any $v_{h,R^{\perp}} \in R_k^{\perp}(K)$, there exists $\btau_1
\in \Sigma_{k+1,h,b}(K)$ such that 
\begin{equation} \label{equ:v-perp}
\div \btau_1 = v_{h,R^{\perp}} \qquad \|\btau_1\|_{\div} \lesssim
\|v_{h,R^{\perp}}\|_0.
\end{equation}
\end{lemma}
\begin{proof}
It follows from Lemma \ref{lem:conforming-bubble} that $\div:
\Sigma_{k+1,h,b}(K) \mapsto R_k^{\perp}(K)$ is onto. Then the quotient
mapping $\tilde{\div}: \Sigma_{k+1,h,b}(K) / \ker(\div) \mapsto
R_k^{\perp}(K)$ is isomorphism. Therefore, there uniquely exists
$\btau_1\in \ker(\div)^{\perp} \cap \Sigma_{k+1,h,b}(K)$ such that 
\[
\div \btau_1 = v_{h,R^{\perp}}. 
\]
It then follows from the definition of $\btau_1$ and scaling argument
that 
\[
\|\btau_1\|_{\div} \lesssim \|\div \btau_1\|_0  = \|v_{h,R^\perp}\|_0.
\]
\end{proof}

%%%%%%%%%%%%%%%%%%%%%%%%%%%%%%%%%%%%%%%% 
%% Stability II: R_k
%%%%%%%%%%%%%%%%%%%%%%%%%%%%%%%%%%%%%%%% 
\subsection{Stability for $R_k$: face-bubble function
spaces}

In light of Lemma \ref{lem:v-perp}, the remaining question for
stability is to solve the rigid motion, namely to find the stable
$\btau_2 \in \Sigma_{k+1,h}$ that $\div_h \btau_2 = v_{h,R}$. Here
$\div_h$ is the $\div$ operator element by element.  And the discrete
$\div$ norm is denoted by 
\[
\|\btau\|_{\div, h}^2 := \sum_{K\in \cT_h} \left(\|\btau
\|_{0,K}^2 + \|\div \btau\|_{0,K}^2 \right) \qquad
\forall \btau \in \Sigma_{k+1,h} \cup \Sigma.
\]
The stability of mixed finite elements for linear elasticity comes
down to the following lemma. 

%% inf-sup lemma-2 
\begin{lemma} \label{lem:inf-sup2}
Assume that $\Sigma_{k+1,h} \subset L^2(\Omega;\mS)$ is any space
equipped with norm $\vertiii{\cdot}$ that satisfies: 
\begin{enumerate} 
\item $\Sigma_{k+1,h,b} \subset \Sigma_{k+1,h}$;
\item $\|\btau\|_{\div,h} \lesssim \vertiii{\btau}$, for all
$\btau \in \Sigma_{k+1,h}$; 
\item $\|\btau\|_{\div,h} \simeq \vertiii{\btau}$, for all
$\btau\in H(\div;\mS)$. 
\end{enumerate} 
Then the following two statements are equivalent: 
\begin{enumerate}
\item For any $v_h \in V_{k,h}$, there exists $\btau_h \in \Sigma_{k+1,h}$
such that 
\begin{equation} \label{equ:inf-sup2}
\div_h \btau_h = v_h \qquad \vertiii{\btau_h} \lesssim \|v_h\|_0.
\end{equation}
\item For any $v_{h,R} \in R_k$, there exists $\btau_2 \in
\Sigma_{k+1,h}$ such that  
\begin{equation} \label{equ:face-motivation} 
\int_{\partial K} \btau_2 \nu \cdot p = \int_K v_{h,R} \cdot p, \quad
\forall p \in R_k(K) \quad \text{and} \quad \vertiii{\btau_2}
\lesssim \|v_{h,R}\|_0.
\end{equation}
\end{enumerate}
Furthermore, a sufficient condition for the above two statements:
there exists a linear operater $\Pi_2: H^1(\Omega;\mS) \mapsto
\Sigma_{k+1,h}$ such that the following diagram is commutative 
\begin{equation} \label{equ:diagram2} 
\begin{CD}
H^1(\Omega;\mS) @> {\div} >> L^2(\Omega;\R^n) \\
@VV \Pi_2 V  @VV P_R V \\
\Sigma_{k+1,h} @> {\div_h} >> R_{k}
\end{CD}
\end{equation}
where $P_R$ is the $L^2$ projection from $L^2(\Omega;\R^n)$ to $R_k$.
\end{lemma}

\begin{proof}
It is easy to check that \eqref{equ:face-motivation} can be derived
from \eqref{equ:inf-sup2} by taking $v_h = v_{h,R}$. On the other
hand, by the stability of continuous formulation (see
\cite{arnold2002mixed, arnold2008finite} for 2D and 3D cases), for
any $v_h \in V_{k,h}$, there exists $\btau \in H^1(\Omega; \mS)$, such
that 
$$
\div \btau = v_h \qquad \|\btau\|_1 \lesssim \|v_h\|_0. 
$$
Let $v_h = v_{h,R} + v_{h,R^\perp} \in R_k \oplus R_k^{\perp}$.  In
light of \eqref{equ:face-motivation}, there exists $\btau_2\in
\Sigma_{k+1,h}$ such that 
$$
\int_{\partial K} \btau_2\nu \cdot p = \int_K v_{h,R} \cdot p, \quad
\forall p \in R_k(K) \quad \text{and} \quad 
\vertiii{\btau_2} \lesssim \|v_{h,R}\|_0 \leq \|v_h\|_0.
$$
Or, 
\[
\int_{K} (\div\btau - \div\btau_2)\cdot p = 0 \qquad \forall p\in
R_k(K).
\]
which yields $v_h - \div_h \btau_2 \in R_k^{\perp}$. Then it
follows from Lemma \ref{lem:v-perp} that there exists $\btau_1\in
\Sigma_{k+1,h,b}$ such that
\[
\div \btau_1 = v_h - \div_h \btau_2 \qquad \|\btau_1\|_{\div} \lesssim
\|\div_h \btau_2 - v_h\|_0 \leq \|\btau_2\|_{\div,h} + \|v_h\|_0
\lesssim \|v_h\|_0. 
\]
Let $\btau_h = \btau_1 + \btau_2$ so that $\div_h \btau_h = v_h$ and  
\[
\vertiii{\btau_h} \lesssim \|\btau_1\|_{\div} +
\vertiii{\btau_2}\lesssim \|v_h\|_0.
\]

For any $v_{h,R} \in R_k$, in light of the stability of continuous
formulation again, there exists $\tilde{\btau} \in H^1(\Omega;\mS)$
such that 
$$ 
\div \tilde{\btau} = v_{h,R} \qquad \|\tilde{\btau}\|_1 \lesssim
\|v_{h,R}\|_0.
$$ 
By taking $\btau_2 = -\Pi_2 \tilde{\btau}$ in the commutative diagram
\eqref{equ:diagram2}, we immediately have 
$$ 
\int_{\partial K} \btau_2 \nu \cdot p = -\int_{K} \div_h(\btau_2)
\cdot p = \int_{K} \div \tilde{\btau} \cdot p = \int_K v_{h,R}\cdot p
  \qquad \forall p\in R_k(K),
$$ 
and 
$$ 
\vertiii{\btau_2} \lesssim \|\Pi_2\| \|\tilde{\btau}\|_1
\lesssim \|v_{h,R}\|_0,
$$ 
which give rise to \eqref{equ:inf-sup2}.
\end{proof}

Lemma \ref{lem:inf-sup2} motivates us to find proper face-bubble
function spaces to meet \eqref{equ:face-motivation}. We will use the
terminology ``recover'', which means finding a suitable face-bubble
function space such that the solution $\btau_2$ of
\eqref{equ:face-motivation} exists.

In light of \eqref{equ:dof-face1}, $\tilde{\Sigma}_{k+1,h,f}(K)$ can
be glued together to obtain the {\em face-bubble function spaces} as
follows. 
\begin{equation} \label{equ:face-bubble-spaces}
\tilde{\Sigma}_{k+1,h,f} := 
\left\{ 
\btau~ \left|
\begin{aligned}
&~\btau|_K \in
\tilde{\Sigma}_{k+1,h,f}(K), \text{~and the moments of~} \btau \nu_F \\
& \text{~up to degree~} k \text{~are continuous across
  the interior faces} 
\end{aligned} 
\right.
\right\}.
\end{equation}
We will prove later that $\tilde{\Sigma}_{k+1,h,f}$ are able to
recover the $\cP_k(F;\R^n)$, which meet the requirement
\eqref{equ:face-motivation} since $R_k|_F \subset \cP_k(F;\R^n)$. The
discussion on the proper subspace of $\tilde{\Sigma}_{k+1,h,f}$ to
meet \eqref{equ:face-motivation} will be given in Section
\ref{sec:reduced}.

%%%%%%%%%%%%%%%%%%%%%%%%%%%%%%%%%%%%%%%%
%%% Space 1: fully P1 space          %%%
%%%%%%%%%%%%%%%%%%%%%%%%%%%%%%%%%%%%%%%%
\subsection{Approximation property option I: nonconforming div-bubble
  function spaces}

Obviously, the spaces $\Sigma_{k+1,h,b}\oplus
\tilde{\Sigma}_{k+1,h,f}$ do not have the approximation property.
Based on the local representation \eqref{equ:local-decomposition}, our
first option is to add the {\em nonconforming div-bubble function
spaces} by combining the $\tilde{\Sigma}_{k+1,h,b}(K)$ element by
element: 
\begin{equation} \label{equ:div-bubble-spaces}
\tilde{\Sigma}_{k+1,h,b} := \{\btau~|~\btau|_K \in
\tilde{\Sigma}_{k+1,h,b}(K), ~\forall K\in\cT_h \}.
\end{equation}
Then we have the following fully nonconforming spaces.

\paragraph{\underline{Fully Nonconforming Spaces}}
The first class of finite element spaces $\Sigma_{k+1,h}^{(1)}$ for
symmetric stress tensors can be written as 
\begin{equation} \label{equ:tensor-space1}
\begin{aligned}
\Sigma_{k+1,h}^{(1)} &:= \Sigma_{k+1,h,b} \oplus
\tilde{\Sigma}_{k+1,h,f} \oplus \tilde{\Sigma}_{k+1,h,b} \\ 
&= \{\btau = \btau_b + \tilde{\btau}_f + \tilde{\btau}_b~|~
\btau_b \in \Sigma_{k+1,h,b}, \tilde{\btau}_f \in
\tilde{\Sigma}_{k+1,h,f}, \tilde{\btau}_b \in
\tilde{\Sigma}_{k+1,h,b}\} \\
&= \{\btau ~|~ \btau|_K \in \cP_{k+1}(K;\mS), \text{ and the moments
  of } \btau\nu_F \\
& \qquad \quad \text{~up to degree~} k \text{~are continuous across
  the interior faces} 
  \}.
\end{aligned}
\end{equation}

The last equality is derived from the following lemma. Let $\mV, \mF,
\mF^i, \mT$ denote, respectively, the number of vertices, faces,
interior faces and simplexes in the triangulation. 

\begin{lemma} \label{lem:space1}
It holds that 
\[
\begin{aligned}
\Sigma_{k+1,h}^{(1)} 
&= \{\btau ~|~ \btau|_K \in \cP_{k+1}(K;\mS),
  \text{~and the moments of~} \btau\nu_F \\
& \qquad \quad \text{~up to degree~} k \text{~are continuous across
  the interior faces} 
  \}.
\end{aligned}
\]
\end{lemma}
\begin{proof}
Denote the right hand side as $\Sigma_{k+1,h}^{(1')}$.  It is easy to
check that $\Sigma_{k+1,h}^{(1)} \subset \Sigma_{k+1,h}^{(1')}$. From
the direct sum of $\tilde{\Sigma}_{k+1,h,f}$, $\Sigma_{k+1,h,b}$ and
$\tilde{\Sigma}_{k+1,h,b}$, we know  
\[
\begin{aligned}
\dim(\Sigma_{k+1,h}^{(1)}) &= \dim(\tilde{\Sigma}_{k+1,h,f}) +
\dim(\Sigma_{k+1,h,b}) + \dim(\tilde{\Sigma}_{k+1,h,b}) \\
&= nC_{n-1+k}^k\mF + \frac{(n+1)n}{2}C_{n+k-1}^{k-1}\mT +
\frac{(n+1)n}{2}C_{k+n-1}^{n-2} \mT, 
\end{aligned}
\]
and 
\[
\dim(\Sigma_{k+1,h}^{(1')}) = \frac{(n+1)n}{2} C_{n+k+1}^{k+1} \mT -
nC_{n-1+k}^k \mF^i.
\]
Then we obtain $\dim(\Sigma_{k+1,h}^{(1)}) =
\dim(\Sigma_{k+1,h}^{(1')})$ by the fact that $\mF + \mF^i = (n+1)\mT$
for the n-dimensional simplicial grids.
\end{proof}

\paragraph{\underline{Degrees of Freedom}}
Based on the property of $\tilde{\Sigma}_{k+1,h,f}$,
$\Sigma_{k+1,h,b}$ and $\tilde{\Sigma}_{k+1,h,b}$, the unisolvent set
of d.o.f. for $\Sigma_{k+1,h}^{(1)}$ is the following set of linear
functionals: 

\begin{subequations} \label{equ:DOF}
\begin{align}
&N_F^{t,m}(\btau) = \langle \btau\nu_F, \varphi_{F,t}e_m \rangle_F &
\text{For all faces~} F \text{~of~} K,
\label{equ:DOF-face}\\
&N_{K}^{\btheta}(\btau) = (\btau, \btheta)_K & \forall \btheta
\in \Sigma_{k+1,h,b}(K) \oplus \tilde{\Sigma}_{k+1,h,b}(K). 
\label{equ:DOF-element}
\end{align}
\end{subequations}

\begin{theorem}
Let $K$ be a simplex in $\R^n$. Any $\btau$ in
$\Sigma_{k+1,h}^{(1)}$ is uniquely determined by the d.o.f. given by
\eqref{equ:DOF}. 
\end{theorem}
\begin{proof}
The local dimension of d.o.f. and $\dim (\cP_{k+1}(K,\mS))$ are both
$\frac{(n+1)n}{2}C_{n+k+1}^{k+1}$. Thus, we only need to show that if
all the d.o.f. applied to $\btau \in \cP_{k+1}(K, \mS)$ vanish, then
$\btau$ vanishes. Let $\btau = \btau_b + \tilde{\btau}_f +
\tilde{\btau}_b \in
\Sigma_{k+1,h,b}(K) \oplus
\tilde{\Sigma}_{k+1,h,f}(K) \oplus 
\tilde{\Sigma}_{k+1,h,b}(K)$, then we immediately obtain
$\tilde{\btau}_f = \bzero$ from Theorem \ref{thm:face-bubble}. Take
$\btheta = \btau_b + \tilde{\btau}_b$ in \eqref{equ:DOF-element} to
find that $\btau = \bzero$.  
\end{proof}

%%%%%%%%%%%%%%%%%%%%%%%%%%%%%%%%%%%%%%%%
%%% Space 2: the P1 Lagrangian Element 
%%%%%%%%%%%%%%%%%%%%%%%%%%%%%%%%%%%%%%%%
\subsection{Approximation property option II: $P_{k+1}(\mS)$ Lagrangian
  Element}
For the purpose of approximation property, the second class of
additional spaces is the standard $P_{k+1}(\mS)$ Lagrangian finite
element space $\Sigma_{k+1,h}^c$ defined in \eqref{equ:lagrangian}.

\paragraph{\underline{Minimal Nonconforming Spaces}}
In most cases, the direct sum property between $\Sigma_{k+1,h}^c$ and
$\tilde{\Sigma}_{k+1,h,f} \oplus \Sigma_{k+1,h,b}$ does not hold. Here
we first modify the face-bubble function spaces
\eqref{equ:face-bubble-spaces} on the boundary as 
\[
\tilde{\Sigma}_{k+1,h,f,0} := \{\btau \in \tilde{\Sigma}_{k+1,h,f} ~|~
\btau \nu = 0,\text{~on~} \cF_h^\partial \}.
\]
Namely, the face-bubble functions related to the boundary are removed.
Then, the second class of finite element spaces $\Sigma_{k+1,h}^{(2)}$
for stress tensors is 
\begin{equation} \label{equ:tensor-space2}
\Sigma_{k+1,h}^{(2)} = \tilde{\Sigma}_{k+1,h,f,0} + (\Sigma_{k+1,h,b}
+ \Sigma_{k+1,h}^c).
\end{equation}

We will prove the direct sum property in lowest order case ($k=0$) for
the {\it strongly regular} grids which are defined as 
\begin{equation} \label{equ:regular-mesh}
\overrightarrow{a_1 a_i} \nparallel \overrightarrow{a_1'a_i} \qquad
\forall F = K\cap K', K=[a_1, a_2, \cdots, a_{n+1}], K'=[a_1', a_2,
\cdots, a_{n+1}].
\end{equation}

%% direct sum property
\begin{lemma} \label{lem:direct-sum}
For the lowest order case ($k=0$), the following holds for the
strongly regular grids:
\begin{equation} \label{equ:direct-sum}
\tilde{\Sigma}_{1,h,f,0} \cap \Sigma_{1,h}^c = \{\bzero\}.
\end{equation}
\end{lemma}
%% proof 
\begin{proof}
Let $\btau \in \tilde{\Sigma}_{1,h,f,0} \cap \Sigma_{1,h}^c$, then 
\[
\btau = \sum_{F\in \cF_h^i} \sum_{l=1}^{n} \beta_{F}^{1,l}
\bphi_{F}^{1,l}.
\]
For any $F = K\cap K' \in \cF_h^i$, $K = [a_1, a_2, \cdots, a_{n+1}]$
and $K' = [a_1', a_2, \cdots, a_{n+1}]$, let $\btheta_F = \sum_{l=1}^n
\beta_{F}^{1,l} \bphi_{F}^{1,l}$, then $\supp(\btheta_F) = K\cup K'$. Note
that $\btheta_F|_K \in \sum_{i=2}^{n+1} \lambda_i t_{1i}t_{1i}^T \R$, then 
\[
\btheta_F|_K = \sum_{i=2}^{n+1} \gamma_{K,i} \lambda_i
t_{1,i}t_{1,i}^T \qquad 
\btheta_F|_{K'} = \sum_{i=2}^{n+1} \gamma_{K',i} \lambda_i t_{1',i}
t_{1',i}^T.
\]
It is easy to see that $\btau \in H^1(\Omega; \mS)$ implies
$[\btheta_F ]|_F = 0$, which yields 
\[
\sum_{i=2}^{n+1} \lambda_i \left\{ 
\gamma_{K,i}(t_{1,i}^T\nu_F) t_{1,i} - \gamma_{K',i}(t_{1',i}^T\nu_F)
t_{1'i} \right\}|_F = 0.
\]
Notice that $\lambda_i, i=2,\cdots, n+1$ are linear independent basis
functions on $F$, $t_{1,i}^T\nu_F \neq 0$ and $t_{1,i} \nparallel
t_{1',i}$ due to the strongly regular assumption, we immediately have
$\gamma_{K,i} = \gamma_{K',i} = 0$. Thus, $\btheta_F = \bzero$ so that
$\btau = \bzero$.
\end{proof}

For the lowest order case on strongly regular grids, the basis
functions of $\Sigma_{1,h}^{(2)}$ can be obtained by the union of
basis functions of $\tilde{\Sigma}_{1,h,f,0}$ \eqref{equ:face-bubble0}
and the standard basis functions of $P_1(\mS)$ Lagrangian element. For
high order elements on general grids, the basis functions and d.o.f.
of $\Sigma_{k+1,h,b} + \Sigma_{k+1,h}^c$ were reported in
\cite{hu2015finite, hu2014family, hu2015family}. At this point, the
union of two sets of basis functions may not be independent, in which
case the iterative methods still work, see
\cite{fortin2000augmented,lee2006convergence}.  

From the analysis in Section \ref{sec:IP-mixed}, any spaces
$\Sigma_{k+1,h}$ that satisfies $\Sigma_{k+1,h}^{(2)} \subset
\Sigma_{k+1,h} \subset \Sigma_{k+1,h}^{(1)}$ can be proved to be
convergent in our framework. Thus, the two classes of finite elements
are the minimal and maximal in this sense, respectively.  Especially
for the lowest order case, the element proposed in \cite{cai2005mixed}
lies in this framework. Below we will give the comparison of the
global dimension of d.o.f. between different spaces in lowest order
case.  

The d.o.f. for $\Sigma_{1,h}^{(1)}$ given in \eqref{equ:DOF} show that the
global dimensions of $\Sigma_{1,h}^{(1)}$ are $3\mT+2\mF$ in 2D and
$12\mT+3\mF$ in 3D. In comparison, the global dimensions of
$\Sigma_{1,h}^{(2)}$ are at most $2\mF^i+3\mV$ in 2D and $3\mF^i+6\mV$
in 3D. We would like to mention that in Cai and Ye's construction
\cite{cai2005mixed}, the global dimensions are  $3\mF$ and $6\mF$ in
2D and 3D, respectively.  The relationship between $\mV, \mF$ and
$\mT$ is $\mV : \mF: \mT \approx 1 : 3: 2$ in 2D case, thus the
proportion of the global dimension of $\Sigma_{1,h}^{(1)}$,
$\Sigma_{1,h}^{(2)}$ and the space in \cite{cai2005mixed} is
approximately $12:9:9$ in 2D case.  In 3D case, however, we have $\mV
: \mF: \mT \approx 1 : 12 : 6$ for the uniform grid. Then the
proportion of the global dimension of $\Sigma_{1,h}^{(1)}$,
$\Sigma_{1,h}^{(2)}$ and Cai and Ye's element is approximately
$108:42:72$ in 3D case.

%% end of file %%
\section{Consistency: Interior Penalty}
\label{sec:IP-mixed}

In this section, we will give the interior penalty mixed
finite element method for the linear elasticity. Without
specification, we will use $\Sigma_{k+1,h}$ to represent the
$\Sigma_{k+1,h}^{(1)}$ defined in \eqref{equ:tensor-space1} or
$\Sigma_{k+1,h}^{(2)}$ defined in \eqref{equ:tensor-space2}, since
both of them are suitable in both the formulation and numerical
analysis.  

%%%%%%%%%%%%%%%%%%%%%%%%%%%%%%%%%%%%%%%%
%%% Formulation %%%
%%%%%%%%%%%%%%%%%%%%%%%%%%%%%%%%%%%%%%%%
\subsection{Interior Penalty Mixed formulation} \label{subsec:IP-mixed}
Our interior penalty mixed method is to find $(\bsigma_h, u_h)\in
\Sigma_{k+1,h} \times V_{k,h}$, such that 
\begin{equation}\label{equ:IP-mixed}
\left\{
\begin{aligned}
&a_h(\bsigma_h,\btau_h) + b_h(\btau_h, u_h) &=&~ 0 \qquad  &\forall
\btau_h \in \Sigma_{k+1,h},\\
&b_h(\bsigma_h, v_h) &=&~ (f, v_h)_{\Omega} \qquad &\forall v_h \in
V_{k,h},
\end{aligned} 
\right.
\end{equation}
where the bilinear forms are defined as 
\begin{subequations} 
\begin{align}
a_h(\bsigma, \btau) &= (\cA \bsigma, \btau)_{\Omega} + \eta \sum_{F\in
  \cF_h^i} h_F^{-1} \int_F [\bsigma] \cdot [\btau] &\forall
  \bsigma, \btau \in \Sigma_{k+1,h} \cup \Sigma,
\label{equ:bilinear-a}\\
b_h(\bsigma, v) &= \sum_{K\in \cT_h} (\div \bsigma, v)_{K} &\forall
\bsigma \in \Sigma_{k+1,h} \cup \Sigma, v \in V_{k,h} \cup V.
\label{equ:bilinear-b} 
\end{align}
\end{subequations}
Here $\eta = \cO(1)$ is a given positive constant.  We then define the
following star norm for $\Sigma_{k+1,h} \cup \Sigma$ as 
\begin{equation} \label{equ:star-norm}
\|\btau \|_{*,h}^2 := \sum_{K\in \cT_h} \left(\|\btau \|_{0,K}^2 +
\|\div \btau\|_{0,K}^2 \right) + \eta \sum_{F\in \cF_h^i}
h_F^{-1}\|[\btau]\|_F^2 \qquad \forall \btau \in \Sigma_{k+1,h} \cup
\Sigma.
\end{equation}

%%%%%%%%%%%%%%%%%%%%%%%%%%%%%%%%%%%%%%%%
%% stability %%
%%%%%%%%%%%%%%%%%%%%%%%%%%%%%%%%%%%%%%%%
\subsection{Stability Analysis}
According to the theory of mixed method, the stability of the saddle
point problem is the corollary of the following two conditions
\cite{brezzi1974existence,brezzi1991mixed}:

\begin{enumerate}
\item K-ellipticity: There exits a constant $C>0$, independent of
the grid size such that 
\begin{equation} \label{equ:K-ellipticity}
a_h(\btau_h, \btau_h) \ge C\|\btau_h\|^2_{*,h}  \qquad \forall
\btau_h \in Z_h,
\end{equation}
where $Z_h = \{\btau_h \in \Sigma_{k+1,h} ~|~ b_h(\btau_h, v_h) = 0,
~\forall v_h \in V_{k,h} \}$. 

\item The discrete inf-sup condition: There exits a constant
$C>0$, independent of the grid size such that 
\begin{equation} \label{equ:inf-sup}
\inf_{v_h \in V_{k,h}} \sup_{\btau_h \in \Sigma_{k+1,h}} \frac{b_h(
\btau_h, v_h)} { \|\btau_h\|_{*,h} \|v_h\|_{0}} \ge C.
\end{equation}
\end{enumerate}

Since $\div_h \Sigma_{k+1,h}\subset V_{k,h}$, we know $\div_h \btau_h
=0$ for any $\btau_h\in Z_h$. This implies the K-ellipticity
\eqref{equ:K-ellipticity}. It remains to show the discrete inf-sup
condition \eqref{equ:inf-sup} in the following lemma.

%% inf-sup lemma-1 %%  
\begin{lemma} \label{lem:inf-sup1}
For any $\btau \in H^1(\Omega;\mS)$, there exists $\btau_h \in
\Sigma_{k+1,h}$ such that
\begin{equation} \label{equ:inf-sup1}
\int_F (\btau - \btau_h)\nu_F \cdot p = 0, \quad \forall p \in
\cP_k(F;\R^n) \quad \text{and} \quad \|\btau_h\|_{*,h} \lesssim
\|\btau\|_{1}.
\end{equation}
\end{lemma}

\begin{proof}
Since $\Sigma_{k+1,h}$ contains the piecewise $\cP_{k+1}$ continuous
functions, we can define a Scott-Zhang \cite{scott1990finite}
interpolation operator $I_h: H^1(\Omega;\mS) \mapsto \{\btau \in
H^1(\Omega;\mS)~|~ \btau|_K \in \cP_{k+1}(K;\mS)\}$ such that  
\[
h_K^{-1}\|\btau - I_h \btau\|_{0,K} + \|\nabla I_h \btau\|_{0,K}
\lesssim \|\nabla \btau\|_{0,K} \qquad \forall K \in \cT_h. 
\]
Define $\btau_h\in \Sigma_{k+1,h}$ as 
\begin{equation} \label{equ:div-projection}
\btau_h = I_h \btau+ \sum_{F\in \cF_h^i}
\sum_{l=1}^n\sum_{s=1}^{C_{n-1+k}^k} \left( \int_F
(\btau- I_h\btau) \nu_F \varphi_{F,s} \cdot e_l \right) \bphi_{F}^{s,l},
\end{equation}
where the face bubble function $\bphi_{F}^{s,l}$ satisfies
$\supp(\bphi_F^{s,l}) = \cT_{h,F}$, and for each $K \in \cT_{h,F}$ is
defined as \eqref{equ:face-bubble-expression}. From the definition of
$\btau_h$, we obtain 
\[
\begin{aligned}
\int_{F'} (\btau_h - I_h\btau) \nu_{F'}\varphi_{F',t} &= 
\sum_{F\in \cF_h^i} \sum_{l=1}^n \sum_{s=1}^{C_{n-1+k}^k} \left(
\int_F (\btau - I_h\btau) \nu_F \varphi_{F,s} \cdot e_l \right) 
\int_{F'} \bphi_{F}^{s,l} \nu_{F'} \varphi_{F',t} \\
&= \sum_{F\in \cF_h^i} \sum_{l=1}^n \sum_{s=1}^{C_{n-1+k}^k} 
\left( \int_F (\btau - I_h\btau) \nu_F\varphi_{F,s} \cdot e_l \right)
e_l \delta_{FF'} \delta{st} \\ 
&= \int_{F'} (\btau - I_h\btau)\nu_{F'}\varphi_{F',t} \quad \forall F'
\in \cF_h^i.
\end{aligned}
\]
and 
\[
\int_{F'} (\btau_h - I_h\btau) \nu_{F'} \cdot p = \int_{F'} (\btau - I_h\btau)
\nu_{F'} \cdot p \qquad \forall F'\in \cF_h^\partial, p \in
\cP_k(F';\R^n),
\]
since Scott-Zhang interpolation operator preserves the boundary
condition. Thus we have 
\[
\int_F (\btau - \btau_h)\nu_F \cdot p = 0 \qquad \forall p \in
\cP_k(F;\R^n).
\]

With the help of Lemma \ref{lem:scaling} and local trace inequality,
\[
\begin{aligned}
\|\btau_h - I_h\btau\|_{\div,h}^2 
& \lesssim \sum_{F\in \cF_h^i} \sum_{l=1}^n \sum_{s=1}^{C_{n-1+k}^k} 
\left| \int_F (\btau - I_h \btau)\nu_F\varphi_{F,s} \cdot e_l\right|^2
\|\bphi_{F}^{s,l}\|_{\div,h,\Omega}^2 \\
& \lesssim \sum_{F\in \cF_h^i} \sum_{l=1}^n \sum_{s=1}^{C_{n-1+k}^k} 
\|(\btau - I_h\btau)\nu_F\|_{0,F}^2 \|\varphi_{F,s}\|_{0,F}^2
\sum_{K'\in \cT_{h,F}} \|\bphi_{F}^{s,l} \|_{\div,K'}^2 \\
& \lesssim \sum_{K\in \cT_h} \sum_{l=1}^n  (h_K^{-1}\|\btau -
  I_h\btau\|^2_{0,K} +h_K |\btau - I_h\btau|_{1,K}^2) h_K^{n-1}
h_K^{-n}\\
& \lesssim \sum_{K\in \cT_h} \sum_{l=1}^n h_K^{-2} \|\btau -
  I_h\btau\|^2_{0,K} + |\btau - I_h\btau|_{1,K}^2 \lesssim
  |\btau|_1^2.
\end{aligned}
\]
And, 
\[
\begin{aligned}
\sum_{F\in \cF_h^i} h_F^{-1}\|[\btau_h]\|_{0,F}^2 
& \lesssim \sum_{F\in \cF_h^i} \sum_{l=1}^n \sum_{s=1}^{C_{n-1+k}^k}
h_F^{-1} \left| \int_F (\btau - I_h\btau)\nu_F \varphi_{F,s} \cdot e_l
\right|^2 \|[\bphi_{F}^{s,l}]\|_{0,F}^2 \\
& \lesssim \sum_{F\in \cF_h^i} \sum_{l=1}^n \sum_{s=1}^{C_{n-1+k}^k} 
h_F^{-1} \|(\btau - I_h\btau)\nu_F\|_{0,F}^2
\|\varphi_{F,s}\|_{0,F}^2 h_F^{-n+1} \\
& \lesssim \sum_{K\in \cT_h} \sum_{l=1}^n h_K^{-1} \left(
    h_K^{-1}\|\btau - I_h\btau\|^2_{0,K} + h_K |\btau -
    I_h\btau|_{1,K}^2 \right) \\
& \lesssim \sum_{K\in \cT_h} \sum_{l=1}^n h_K^{-2} \|\btau -
  I_h\btau\|^2_{0,K} + |\btau - I_h\btau|_{1,K}^2 \lesssim
  |\btau|_1^2.
\end{aligned}
\]
Then we have 
\[
\|\btau_h\|_{*,h} \leq \|\btau_h - I_h\btau\|_{*, h} +
\|I_h\btau\|_{\div,h} \lesssim \|\btau\|_1. 
\]
\end{proof}

Essentially, we define an operator $\Pi_{h,f}^{\div,*}: H^1(\Omega;
\mS) \mapsto \Sigma_{k+1,h}$ in the construction
\eqref{equ:div-projection} as 
$$
\Pi_{h,f}^{\div,*} \btau := I_h \btau + \sum_{F\in \cF_h^i}
\sum_{l=1}^n \sum_{s=1}^{C_{n-1+k}^k} \left( \int_F (\btau - I_h\btau)
\nu_F\varphi_{F,s} \cdot e_l \right) \bphi_{F}^{s,l}.
$$
Then we know that $\div_h \Range(I - \Pi_{h,f}^{\div,*}) \subset
R_k^{\perp}$. Let $\div_h^{-1}(R_k^{\perp}) := \{\btau \in
\Sigma_{k+1,h,b}~|~ \div_h \btau \in R_k^{\perp}\}$, then Lemma
\ref{lem:v-perp} implies a stable linear operator
$\Pi_{h,b}^{\div,*}: \div_h^{-1}(R_k^{\perp}) \mapsto
\Sigma_{k+1,h,b}$. Define $\Pi_{h}^{\div,*} := \Pi_{h,f}^{\div,*} +
\Pi_{h,b}^{\div,*} (I - \Pi_{h,f}^{\div,*})$, we immediately have the
following commutative diagram: 
\begin{equation} \label{equ:diagram} 
\begin{CD}
H^1(\Omega;\mS) @> {\div} >> L^2(\Omega;\R^n) \\
@VV \Pi_h^{\div,*} V  @VV \Pi_h^0 V \\
\Sigma_{k+1,h} @> {\div_h} >> V_{k,h}
\end{CD}
\end{equation}
where $\Pi_h^0$ is the $L^2$ projection operator on $V_{k,h}$.  In
summary, we have the following theorem.

\begin{theorem} \label{thm:stability}
For any $f \in L^2(\Omega,\R^n)$, the discrete variational problem
\eqref{equ:IP-mixed} is well-posed for $(\Sigma_{k+1,h},
\|\cdot\|_{*,h})$ and $(V_{k,h}, \|\cdot\|_0)$. 
\end{theorem}
\begin{proof}
It follows from the definition of $\|\cdot\|_{*,h}$ that it is
stronger than $\|\cdot\|_{\div,h}$. Notice that $R_k|_F \subset
\cP_k(F;\R^n)$, we immediately obtain that \eqref{equ:face-motivation}
in Lemma \ref{lem:inf-sup2} is satisfied, which implies the stability
of the finite elements. 
\end{proof}

%%%%%%%%%%%%%%%%%%%%%%%%%%%%%%%%%%%%%%%%
%% error estimate %%
%%%%%%%%%%%%%%%%%%%%%%%%%%%%%%%%%%%%%%%%
\subsection{Error Estimate}

Let $(\bsigma, u) \in \Sigma \times V$ be the exact solution of
\eqref{equ:elasticity}, then  
\begin{equation} \label{equ:error-equ}
\left\{
\begin{aligned}
a_h(\bsigma - \bsigma_h, \btau_h) + b_h(\btau_h, u-u_h) &= \langle
[\btau_h], u\rangle_{\cF_h^i} \qquad & \forall \btau_h \in
\Sigma_{k+1,h}, \\
b_h(\bsigma - \bsigma_h, v_h) &= 0 \qquad & \forall v_h \in V_{k,h},
\end{aligned}
\right.
\end{equation}
where $\langle [\btau_h], u\rangle_{\cF_h^i} = \sum_{F\in \cF_h^i}
\int_F [\btau_h] \cdot u$ is the {\em consistency error}.  From the
well-posedness of the discrete variational problem
\eqref{equ:IP-mixed} and the error estimate by Babu\v{s}ka
\cite{babuvska1971error}, we have the following theorem. 

%% Babuska estimate 
\begin{theorem} \label{thm:estimate1}
For any $f\in L^2(\Omega, \R^n)$, let $(\bsigma, u) \in \Sigma \times
V$ be the exact solution of problem \eqref{equ:elasticity} and
$(\bsigma_h, u_h) \in \Sigma_{k+1,h} \times V_{k,h}$ be the finite
element solution of \eqref{equ:IP-mixed}. Then 
\begin{equation} \label{equ:estimate1}
\|\bsigma - \bsigma_h\|_{*,h} + \|u - u_h\|_0 \lesssim
\inf_{\substack{\btau_h \in \Sigma_{k+1,h}\\ v_h\in V_{k,h}}} 
(\|\bsigma - \btau_h\|_{*,h} + \|u - v_h\|_0) 
+ \sup_{\btau_h \in \Sigma_{k+1,h}} \frac{|\langle [\btau_h], u
\rangle_{\cF_h^i}|}{\|\btau_h\|_{*,h}}. 
\end{equation}
\end{theorem}

\begin{proof}
Define the bilinear form 
\[
\tilde{a}_h((\bsigma, u)^T, (\btau, v)^T) \triangleq 
a_h(\bsigma, \btau) + b_h(\btau, u) - b_h(\bsigma, v),
\]
which satisfies the inf-sup condition on $\Sigma_{k+1,h} \times
V_{k,h}$ due to the Theorem \ref{thm:stability}. Therefore, for any
$(\btheta_h, w_h)^T \in \Sigma_{k+1,h} \times V_{k,h}$, 
\[
\begin{aligned}
\|\btheta_h - \bsigma_h\|_{*,h} + \|w_h - u_h\|_0  
&\lesssim \sup_{\substack{\btau_h \in \Sigma_{k+1,h}\\ v_h \in
V_{k,h}}}
\frac{\tilde{a}_h((\btheta_h - \bsigma_h, w_h - u_h)^T, (\btau_h,
v_h)^T)}{\|\btau\|_{*,h} + \|v_h\|_0} \\
&= \sup_{\substack{\btau_h \in \Sigma_{k+1,h}\\ v_h \in V_{k,h}}}
\frac{\tilde{a}_h((\btheta_h - \bsigma, w_h - u)^T, (\btau_h,
v_h)^T) + \langle [\btau_h], u\rangle_{\cF_h^i}}{\|\btau_h\|_{*,h} +
\|v_h\|_0} \\
&\lesssim \|\btheta_h - \bsigma\|_{*,h} + \|w_h - u\|_0 +
\sup_{\btau_h \in \Sigma_{k+1,h}} \frac{|\langle [\btau_h], u
\rangle_{\cF_h^i}|}{\|\btau_h\|_{*,h}}.
\end{aligned}
\]
The desired result \eqref{equ:estimate1} then follows from the
triangle inequality.
\end{proof}

For the consistency error, we have the following lemma.  

\begin{lemma} \label{lem:consistent-error}
Assume that $u \in H^{k+1}(\Omega; \R^n)$, it holds that 
\begin{equation} \label{equ:consistent-error}
\sup_{\btau_h \in \Sigma_{k+1,h}} \frac{|\langle [\btau_h],
u\rangle_{\cF_h^i}|}{\|\btau_h\|_{*,h}} \lesssim h^{k+1} |u|_{k+1}.
\end{equation}
\end{lemma}
\begin{proof}
For any $\btau_h \in \Sigma_{k+1,h}$, it follows from the Poincar\'{e}
inequality and standard scaling argument that  
\[
\begin{aligned}
|\sum_{F\in\cF_h^i} \int_F [\btau_h] \cdot u|
& = | 
\sum_{F\in\cF_h^i} \inf_{p \in \cP_k(F;\R^n)} \int_F [\btau_h]
\cdot (u-p) | \\
& \lesssim \sum_{F\in \cF_h^i} \|[\btau_h]\|_{0,F} \inf_{p \in
  \cP_k(F;\R^n)} \|u-p\|_{0,F} \\
& \lesssim \sum_{F\in \cF_h^i} \|[\btau_h]\|_{0,F} h_F^{k+1/2}
|u|_{k+1,\cT_{h,F}} \\
& \lesssim \left(\sum_{F\in \cF_h^i}
h_F^{-1}\|[\btau_h]\|_{0,F}^2 \right)^{1/2} \left(\sum_{F\in
\cF_h^i} h_F^{2(k+1)}|u|_{k+1,\cT_{h,F}}^2  \right)^{1/2} \\ 
& \lesssim h^{k+1} \|\btau_h\|_{*,h} |u|_{k+1}.
\end{aligned}
\]
\end{proof}

We have the following approximation property of the finite element
spaces. 
\begin{lemma} \label{lem:approximation}
Assume that $\bsigma \in H^{k+2}(\Omega; \mS)$, $u \in H^{k+1}(\Omega;
\R^n)$, then 
\begin{subequations} 
\begin{align}
\inf_{\btau_h \in \Sigma_{k+1}} \|\bsigma - \btau_h\|_{*,h} &\lesssim
h^{k+1} |\bsigma|_{k+2}, \label{equ:approx-sigma}\\
\inf_{v_h \in V_{k,h}} \|u - v_h\|_0 & \lesssim h^{k+1}|u|_{k+1}.
\label{equ:approx-u}
\end{align}
\end{subequations}
\end{lemma}

\begin{proof}
The approximation \eqref{equ:approx-sigma} follows from 
\[
\inf_{\btau_h \in \Sigma_{k+1,h}} \|\bsigma - \btau_h\|_{*,h} \leq 
\|\bsigma - I_h \bsigma \|_{\div, h} \leq h^{k+1} |\bsigma|_{k+2},
\]
since the Scott-Zhang interpolation operator $I_h$ preserves symmetric
$\cP_{k+1}$ functions locally. The approximation property of $V_h$ can
be proved by taking $v_h = \Pi_h^0 u$ on the left side of
\eqref{equ:approx-u}.
\end{proof}

In light of Theorem \ref{thm:estimate1}, Lemma \ref{lem:approximation}
and Lemma \ref{lem:consistent-error}, we have the following error
estimate.

\begin{theorem} \label{thm:estimate2}
Assume that the exact solution of problem \eqref{equ:elasticity}
satisfies $\bsigma \in H^{k+2}(\Omega; \mS)$, $u \in H^{k+1}(\Omega;
\R^n)$. Then
\begin{equation} \label{equ:estimate2}
\|\bsigma - \bsigma_h\|_{*,h} + \|u - u_h\|_0 \lesssim h^{k+1}
(|\bsigma|_{k+2} + |u|_{k+1}).
\end{equation}
\end{theorem}

\section{Discussion and Reduced Elements} \label{sec:reduced}

In the proof of Theorem \ref{thm:stability}, we use the fact that
$R_k|_F \subset \cP_k(F;\R^n)$. Actually, the normal components of
face-bubble functions are only needed to recover the moments of rigid
motion on each face. Notice that 
\begin{equation} \label{equ:rigid-skew}
R_k(K) =
\begin{cases}
\R^n & \quad k=0, \\ 
\R^n + \mK x & \quad k\geq 1,
\end{cases}
\end{equation}
where $\mK$ represents the space of real skew-symmetric matrices of
order $n\times n$. This means the rigid motion on each face are at
most linear. This observation gives us some space to reduce the
dimension of face-bubble function spaces.

In light of Lemma \ref{lem:inf-sup2}, the remaining question is how to
pick up some face-bubble functions in $\tilde{\Sigma}_{k+1,h,f}$ to
recover the moments of $R_k|_F$.  For the lowest order case, $R_0|_F =
\cP_0(F;\R^n)$ and $\dim(\tilde{\Sigma}_{1,h,f}(K)\nu_F|_F) = n$,
which means that our nonconforming finite elements are optimal and the
interior penalty term has to be added. For the higher order case
$k\geq 1$, $R_k|_F \subset \cP_1(F;\R^n)$.  Traditionally, it suffices
to recover the normal component of stress up to moments of
$\cP_1(F;\R^n)$ to make the elements stable. Table
\ref{tab:n-bubble-2D} and \ref{tab:n-bubble-3D} illustrates the
dimension of $R_k|_F$, $\cP_1(F;\R^n)$, $\btau \nu|_F$ of
$\tilde{\Sigma}_{k+1,h,f}$ and $\Sigma_{k+1,h,f}^c$ in 2D and 3D. We
would like to emphasize that the $H^1(\mS)$ face-bubble function
spaces $\Sigma_{k+1,h,f}^c$ satisfiy 
\begin{equation} \label{equ:continuous-face-bubble2}
\Sigma_{k+1,h,f}^c \subset \tilde{\Sigma}_{k+1,h,f} \cap
\Sigma_{k+1,h}^c.
\end{equation}

\begin{table}[!ht] 
\small
\centering
\begin{tabular}{c|c|c|c|c}
  \hline
  $k$ & $R_k|_F$ & $\cP_0(F;\R^2)$ or $\cP_1(F;\R^2)$ & $\btau
  \nu|_F$ of $\tilde{\Sigma}_{k+1,h,f}$ & $\btau \nu|_F$ of
  $\Sigma_{k+1,h,f}^c$ 
  \\ \hline
  $k=0$ & 2 & 2 & 2 & 0 \\ \hline
  $k=1$ & 3 & 4 & 4 & 2 \\ \hline 
  $k=2$ & 3 & 4 & 6 & 4 \\ \hline
\end{tabular}
\caption{The dimension of $R_k|_F$, $\btau \nu|_F$ of
  $\tilde{\Sigma}_{k+1,h,f}$ and $\Sigma_{k+1,h,f}^c$ in 2D}
\label{tab:n-bubble-2D}
\end{table}

\begin{table}[!ht]
\small
\centering
\begin{tabular}{c|c|c|c|c}
  \hline
  $k$ & $R_k|_F$ & $\cP_0(F;\R^3)$ or $\cP_1(F;\R^3)$ & $\btau \nu|_F$
  of $\tilde{\Sigma}_{k+1,h,f}$ & $\btau \nu|_F$ of
  $\Sigma_{k+1,h,f}^c$ \\ \hline
  $k=0$ & 3 & 3 & 3 & 0 \\ \hline
  $k=1$ & 6 & 9 & 9 & 0 \\ \hline 
  $k=2$ & 6 & 9 & 18 & 3 \\ \hline 
  $k=3$ & 6 & 9 & 30 & 9 \\ \hline 
\end{tabular}
\caption{The dimension of $R_k|_F$, $\btau \nu|_F$ of
  $\tilde{\Sigma}_{k+1,h,f}$ and $\Sigma_{k+1,h,f}^c$ in 3D}
\label{tab:n-bubble-3D}
\end{table}

We can observe that the $H^1(\mS)$ face-bubble function is not enough
to do the job when $k \leq n-1$. A natural question: can we pick up
some $H(\div)$ conforming functions of degree $k+1$ whose normal
component will recover the $\cP_1(F;\R^n)$? For general grids, the
answer is negative when $1\leq k \leq n-1$.

%% Lemma: maximal face bubble %% 
\begin{lemma} \label{lem:maximal-face-bubble}
Given any $K = [a_1, \cdots, a_{n+1}]$ and $F_l \subset \partial K$.
For $\btau \in \cP_{k+1}(K;\mS)$, 
$$ 
\btau \nu|_{F_s} = 0 \qquad \forall s\neq l,
$$ 
is equivalent to 
\begin{equation} \label{equ:maximal-face-bubble} 
\btau \in \sum_{j=1,j\neq l}^{n+1} \lambda_j
\cP_k^{\hat{0},\hat{l}}(K;\R) t_{lj}t_{lj}^T \oplus \sum_{1\leq i < j
\leq n+1} \lambda_i \lambda_j \cP_{k-1}(K;\R) t_{ij}t_{ij}^T.
\end{equation} 
\end{lemma}

\begin{proof}
We only prove the case that $l=1$ for simplicity. Denote 
$$ 
\btau = \sum_{1\leq i < j \leq n+1} p_{ij} t_{ij}t_{ij}^T \qquad
p_{ij} \in \cP_{k+1}(K;\R).
$$
It follows from $\btau \nu|_{F_s} = 0 ~(\forall s\neq 1)$ that 
$$ 
\sum_{1\leq i < j \leq n+1} p_{ij} t_{ij}t_{ij}^T \nu_{F_s} =
\sum_{j\neq s} p_{sj}t_{sj} (t_{sj}^T \nu_{F_s}) = 0 \qquad s = 2,
\cdots, n+1, 
$$
which yields $p_{sj}|_{F_s} = 0 ~(j \neq s)$, since $t_{sj}^T
\nu_{F_s} \neq 0$ and $\{t_{sj}, j\neq s\}$ are linearly independent.
Therefore, 
\begin{equation} \label{equ:decomposition-coef}
p_{ij} = \begin{cases}
\lambda_i \lambda_j \tilde{p}_{ij} \in \lambda_i \lambda_j
\cP_{k-1}(K;\R) & 2 \leq i < j \leq n+1, \\
\lambda_j \bar{p}_{1j} \in \lambda_j \cP_k(K;\R) & i = 1,
2\leq j \leq n+1. 
\end{cases} 
\end{equation}
From \eqref{equ:decomposition-remark}, $\bar{p}_{1j}$ in
\eqref{equ:decomposition-coef} can be decomposed as $\bar{p}_{1j} =
\tilde{p}^{\hat{0},\hat{1}}_j + \lambda_1 \tilde{p}_{1j} \in
\cP_k^{\hat{0},\hat{1}}(K;\R) \oplus \lambda_1 \cP_{k-1}(K;\R)$ and
consequently  
$$ 
\btau = \sum_{j=2}^{n+1} \lambda_j \tilde{p}_j^{\hat{0},\hat{1}}
t_{0j}t_{0j}^T + \sum_{1\leq i < j \leq n+1} \lambda_i \lambda_j
\tilde{p}_{ij} t_{ij}t_{ij}^T. 
$$ 
On the other hand, it is easy to check that
\eqref{equ:maximal-face-bubble} implies $\btau\nu|_{F_s} = 0
~(s=2,\cdots,n+1)$. Then we finish the proof.
\end{proof}
%% end of proof %% 

%% Thm: Optimal %%
\begin{theorem} \label{thm:optimal}
Given any interior face $F = [a_2, \cdots, a_{n+1}] = K \cap K'$, $K
= [a_1, a_2, \cdots, a_{n+1}]$ and $K' = [a_1', a_2, \cdots,
a_{n+1}]$.  Suppose $\forall \{i_1, \cdots, i_s\} \subset \{2, \cdots,
  n+1\}, s\leq n-1$, 
\begin{equation} \label{equ:grid-assumption} 
[a_1, a_{i_1}, \cdots, a_{i_s}], \text{~and~} [a_1', a_{i_1}, \cdots,
  a_{i_s}] \text{~are not in the~} s-\dim \text{~hyperplane}, 
\end{equation}
then it is impossible to pick the $H(\div, K\cup K'; \mS)$ of degree
$k+1$ conforming face bubble functions to recover the moments of
$\cP_1(F;\R^n)$ when $k \leq n-1$.
\end{theorem}

\begin{proof}
For any face bubble function $\btau \in H(\div, K\cup K'; \mS)$, from
Lemma \ref{lem:maximal-face-bubble} we know 
$$ 
\btau \nu|_F \in \sum_{j=2}^{n+1} \lambda_j \cP_k(F;\R) t_{1j} \cap
\sum_{j=2}^{n+1} \lambda_j \cP_k(F;\R) t_{1'j}.
$$ 
Moreover, it can be written in the following form
$$
\btau \nu|_{F} = \sum_{|\boldsymbol{m}|=k+1}  c_{\boldsymbol{m}}
\lambda_{2}^{m_2} \cdots \lambda_{n+1}^{m_{n+1}} \qquad m_i \geq 0, 
$$
where $c_{\boldsymbol{m}}$ is the coefficient vector. We collect the
monomial terms of $\btau \nu|_{F}$ in the following two cases:
\begin{enumerate} 
\item There exists $i$ such that $m_i = 0$. Thus, at least one term of
$\lambda_2, \cdots, \lambda_{n+1}$ does not appear, then the
coefficient $c_{\boldsymbol{m}}$ belongs to $\text{span}\{t_{1i_1},
\cdots, t_{1i_s}\} \cap \text{span}\{t_{1'i_1}, \cdots, t_{1'i_s}\}$,
which lies in $F_0$ by the assumption \eqref{equ:grid-assumption}.  In
this case, $c_{\boldsymbol{m}} \cdot \nu_F = 0$.  
\item $\alpha_i >
0$ for all $i = 1,2, \cdots, n$. Since $k\leq n-1$, the only possible
term is $\lambda_1 \lambda_2 \cdots \lambda_n$ with a constant vector
as coefficient.
\end{enumerate}
Therefore, 
$$ 
\btau \nu|_F = c_{1,1,\cdots,1} \lambda_2 \cdots \lambda_{n+1}.
$$ 
Namely, 
$$ 
\dim\left( \{\btau\nu|_F~|~ \btau ~\text{is}~H(\div, K\cup
K')~\text{conforming face-bubble}\}\right) \leq 1 <
\dim(\cP_1(F;\R^n)),
$$ 
which means the normal components of conforming face-bubble functions
can not recover the moments of $\cP_1(F;\R^n)$ when $k \leq n-1$.
\end{proof}

%%% Remark: extended to R_k %%%
\begin{remark}
Theorem \ref{thm:optimal} admits $R_k|_{F}$ since the dimension of its
normal components can be great than 1 for $1\leq k \leq n-1$.  
\end{remark}

From Theorem \ref{thm:optimal}, the nonconforming finite elements of
degree $k+1$ have to be used to construct the face-bubble function
spaces when $k \leq n-1$. Let $\{\varphi_{F,t}, t = 1, \cdots, n\}$ be
a basis of $\cP_1(F;\R)$. Then we only need the
corresponding $n^2$ face-bubble functions $\bphi_{F}^{s,t} (s,l =
1,\cdots,n)$ defined in \eqref{equ:face-bubble-expression} to
recover the moments of $\cP_1(F;\R^n)$. These elements reduce the
local dimension of nonconforming face-bubble functions from
$nC_{n-1+k}^k$ to $n^2$, which does work when $2\leq k \leq n-1$.

%%% Discussion on Hu's element %%%
For the case that $k \geq n$, one of the significant results proposed
by Hu \cite{hu2015finite} is that the $H^1(\mS)$ face-bubble functions
can recover the moments of $\cP_1(F;\R^n)$, which can be seen from
the second case in the proof of Theorem \ref{thm:optimal}.  Precisely,
on face $F = [a_2, \cdots, a_{n+1}]$, the normal components of
$H^1(\mS)$ face-bubble functions $\lambda_2 \cdots \lambda_{n+1}
\cP_1(K;\R) \mS\nu_F$ will recover the moments of $\cP_1(F;\R^n)$
since $\mS\nu_F = \R^n$. Thanks to the $H(\div)$ conformity of the
$H^1(\mS)$ face-bubble functions, the interior penalty term is
degenerated to be zero consequently. On the other hand, Theorem
\ref{thm:optimal} also indicates that we need enough $H(\div)$ bubble
functions that contain the factor $\lambda_2 \cdots \lambda_{n+1}$ to
satisfies \eqref{equ:inf-sup2}. 

%% lemma: Hu-optimal %% 
\begin{lemma}
Given any interior face $F = [a_2, \cdots, a_{n+1}] = K \cap K'$, $K =
[a_1, a_2, \cdots, a_{n+1}]$ and $K' = [a_1', a_2, \cdots, a_{n+1}]$.
For any $\btau \in H(\div, K\cup K'; \mS) \cap \cP_{n+1}(K\cup
K';\mS)$ that 
\begin{equation} \label{equ:Hdiv-1-n} 
\btau|_K, \btau|_{K'} \in \lambda_2\lambda_2 \cdots \lambda_{n+1}
\cP_1(\mS),
\end{equation}
then there exists $\btheta \in H^1(K\cup K';\mS) \cap \cP_{n+1}(K\cup
K';\mS)$ such that $\btau - \btheta \in \Sigma_{n+1,h,b}$.
\end{lemma}
\begin{proof}
Due to the $H(\div)$ conformity of $\btau$, we know that 
$$ 
\btau \nu|_F \in \lambda_2 \cdots \lambda_{n+1} \cP_1(F;\R^n).
$$ 
Since $\mS \nu_{F} = \R^n$, there exists $\btheta \in H^1(K\cup
K';\mS) \cap \cP_{n+1}(K\cup K';\mS)$ such that 
$$ 
\btheta \in \lambda_2 \cdots \lambda_{n+1} \cP_1(\mS)\quad \text{and}
\quad \btheta \nu|_F = \btau \nu|_F.
$$ 
Thus, 
$$ 
(\btau - \btheta)\nu|_F = 0 \qquad \text{and} \quad (\btau -
\btheta)|_{F_i} = 0, \quad i = 2, \cdots, n+1,
$$ 
which yields $\btau - \btheta \in \Sigma_{n+1,h,b}$.
\end{proof}

The above lemma means that $H(\div;\mS)$ face-bubble functions in the
form of \eqref{equ:Hdiv-1-n} can be derived by the combination of
$H^1(\mS)$ face bubble function and proper div-bubble function. In
this sense, the finite elements proposed by Hu \cite{hu2015finite} are
optimal for the case that $k \geq n$.

\section{Numerical results} \label{sec:numerical}

In this section, we give the numerical results for both 2D and 3D
cases. The simulation is implemented using the MATLAB software package
$i$FEM \cite{iFEM}. The compliance tensor in our computation is 
\[
\cA\bsigma = \frac{1}{2\mu} \left(\bsigma - \frac{\lambda}{2\mu +
n\lambda} \tr(\bsigma)\bI_n \right),
\]
where $\bI_n \in \R^{n\times n}$ is the identity matrix. The Lam\'{e}
constants are set to be $\mu = 1/2$ and $\lambda = 1$. 
%%The algebraic accuracy of quadrature rules in our numerical tests is set to be 2.

\subsection{2D Test}
The 2D displacement problem is computed on the unit square $\Omega =
[0,1]^2$ with a homogeneous boundary condition that $u = 0$ on
$\partial \Omega$. Let the exact solution be 
\[
u = 
\begin{pmatrix}
\re^{x-y} xy(1-x)(1-y) \\
\sin(\pi x)\sin(\pi y)
\end{pmatrix}.
\]
The exact stress function $\bsigma$ and the load function $f$ can be
analytically derived from \eqref{equ:elasticity} for a given $u$.

The uniform grids with different grid sizes are applied in the
computation. We would like to emphasize that the uniform grids satisfy
the strongly regular assumption \eqref{equ:regular-mesh} so that the
discrete systems when applying $\Sigma_{1,h}^{(2)}$ for stress can be
solved by direct solver, for example Matlab backslash solver. The
parameter $\eta$ in \eqref{equ:bilinear-a} is set to be $1$ in the 2D
test.

\begin{table}[!ht] 
\small
\centering
%\begin{tabular}{c|cc|cc|cc|cc|cc}
\begin{tabular}{p{.4cm}|p{1.3cm}p{.4cm}|p{1cm}p{.4cm}|p{1.1cm}p{.4cm}|p{1.2cm}p{.4cm}|p{.8cm}p{1cm}}
  \hline
  $1/h$	&$\|u -u_h\|_0$	& $h^n$	&$\|\bepsilon_h\|_0$&$h^n$
  &$\|\div_h\bepsilon_h\|_0$	&$h^n$ &
  $\|[\bsigma_h]\|_{0,\cF_h^i}$ & $h^n$	& $\dim V_{0,h}$ & $\dim
  \Sigma_{1,h}^{(1)}$  \\ \hline
%  2		&0.27066 &--		&1.08762 &--		&7.27063	&--	  & 0.14760 &--   &16		&56 \\ 
%  4		&0.13585 &0.99	&0.41629 &1.39	&3.82153	&0.93	& 0.09439 &0.64 &64		&208 \\
%  8		&0.06731 &1.01	&0.17195 &1.28	&1.93423	&0.98	& 0.03804 &1.31 &256 	&800 \\
  8		&0.06731 &--	&0.17195 &--	&1.93423	&--	& 0.03804 &-- &256 	&800 \\
  16	&0.03355 &1.00	&0.07954 &1.11	&0.97005	&1.00	& 0.01391 &1.45 &1024	&3136	\\
  32	&0.01676 &1.00	&0.03886 &1.03  &0.48539	&1.00	& 0.00496 &1.49 &4096	&12416	\\
  64	&0.00838 &1.00	&0.01931 &1.01	&0.24274	&1.00	& 0.00176 &1.50 &16384	&49408	\\
%  128	&0.00419 &1.00	&0.00964 &1.00	&0.12138	&1.00	& 0.00062 &1.50 &65536	&197120	\\
  \hline
\end{tabular}
\caption{The error, $\bepsilon_h = \bsigma - \bsigma_h$, and
  convergence order for 2D on uniform grids, $\Sigma_{1,h}^{(1)}$}
\label{tab:numerical-2D-1}
\end{table}

\begin{table}[!ht] 
\small
\centering
%\begin{subtable}{
\begin{tabular}{p{.4cm}|p{1.3cm}p{.4cm}|p{1cm}p{.4cm}|p{1.1cm}p{.4cm}|p{1.2cm}p{.4cm}|p{.8cm}p{1cm}}
  \hline
  $1/h$	&$\|u -u_h\|_0$	&$h^n$	&$\|\bepsilon_h\|_0$&$h^n$
  &$\|\div_h\bepsilon_h\|_0$	&$h^n$ &
  $\|[\bsigma_h]\|_{0,\cF_h^i}$ & $h^n$	& $\dim V_{0,h}$ & $\dim
  \Sigma_{1,h}^{(2)}$  \\ \hline
%%  2		&0.34592 &--		&1.17600 &--		&7.27063	&--	  & 0.29442 &--   &16		&43 \\ 
%%  4		&0.18589 &0.90	&0.59236 &0.99	&3.82153	&0.93	& 0.15480 &0.93 &64		&155 \\
%%  8		&0.11497 &0.69	&0.27495 &1.11	&1.93423	&0.98	& 0.08925 &0.79 &256 	&595 \\
  8		&0.11497 &--	&0.27495 &--	&1.93423	&--	& 0.08925 &-- &256 	&595 \\
  16	&0.06714 &0.78	&0.10042 &1.45	&0.97005  &1.00	& 0.04116 &1.12 &1024	&2339	\\
  32	&0.03578 &0.91	&0.03294 &1.61  &0.48539	&1.00	& 0.01613 &1.35 &4096	&9283	\\
  64	&0.01832 &0.97	&0.01066 &1.63	&0.24274	&1.00	& 0.00593 &1.44 &16384 &36995	\\
%  128	&0.00924 &0.99	&0.00352 &1.60	&0.12138	&1.00	& 0.00213 &1.48 &65536 &147715	\\
  \hline
\end{tabular}
\caption{The error, $\bepsilon_h = \bsigma - \bsigma_h$, and
  convergence order for 2D on uniform grids, $\Sigma_{1,h}^{(2)}$}
\label{tab:numerical-2D-2}
%}
\end{table}

\begin{table}[!ht] 
\small
\centering
\begin{tabular}{p{.4cm}|p{1.3cm}p{.4cm}|p{1cm}p{.4cm}|p{1.1cm}p{.4cm}|p{1.2cm}p{.4cm}|p{.8cm}p{1cm}}
  \hline
  $1/h$	&$\|u -u_h\|_0$	&$h^n$	&$\|\bepsilon_h\|_0$&$h^n$
  &$\|\div_h\bepsilon_h\|_0$	&$h^n$ &
  $\|[\bsigma_h]\|_{0,\cF_h^i}$ & $h^n$	& $\dim V_{0,h}$ & $\dim
  \Sigma_{1,h}^{(2)}$  \\ \hline
%  2		&0.36574 &--		&0.84039 &--		&5.45445	&--	  & 0.41495 &--   &22		&57 \\ 
%  4		&0.15149 &1.27	&0.31838 &1.40	&3.01702	&0.85	& 0.15347 &1.44 &88		&210 \\
%  8		&0.07784 &0.96	&0.13044 &1.29	&1.53835	&0.97	& 0.06441 &1.25 &352 	&813 \\
  8		&0.07784 &--	&0.13044 &--	&1.53835	&--	& 0.06441 &-- &352 	&813 \\
  16	&0.04108 &0.92	&0.05275 &1.31	&0.77269  &0.99	& 0.02627 &1.29 &1408	&3207	\\
  32	&0.02142 &0.94	&0.01988 &1.41  &0.38678	&1.00	& 0.01014 &1.37
      &5632	&12747	\\
  64	&0.01097 &0.97	&0.00724 &1.46	&0.19344	&1.00	& 0.00375 &1.44 &22528 &50835	\\
%  128	&0.00555 &0.98	&0.00262 &1.47	&0.09672	&1.00	& 0.00136 &1.46
%  &90112 &203043	\\
  \hline
\end{tabular}
\caption{The error, $\bepsilon_h = \bsigma - \bsigma_h$, and
  convergence order for 2D on unstructured grids, $\Sigma_{1,h}^{(2)}$}
\label{tab:numerical-2D-unstructured}
\end{table}

First, we use $\Sigma_{1,h}^{(1)}$ for the stress approximation. The
errors and the convergence order in various norms are listed in Table
\ref{tab:numerical-2D-1}. The first order convergence is observed for
both displacement and stress. The $L^2$ error of the stress jump on
interior edge is convergent with order $1.5$, as the theoretical
error estimate \eqref{equ:estimate2}. When applying $\Sigma_{1,h}^{(2)}$
for the stress approximation, the dimension of $\Sigma_{1,h}$ has been
reduced by approximately 25\%, see Table \ref{tab:numerical-2D-2}.
To our supervise, the convergence order of $L^2$ error for stress is
much higher than the error estimate \eqref{equ:estimate2} when using
$\Sigma_{1,h}^{(2)}$.  The phenomenon can also be observed on the
uniformly refined unstructured grids, see Table
\ref{tab:numerical-2D-unstructured}. 
%%It is probably ascribed to the superconvergence of $L^2$ error for
%%stress, whose theoretical explanation is still an open problem.  

In Table \ref{tab:numerical-2D-2-k1}, we list the errors of
$\bsigma_h$ and $u_h$ with finite element spaces $\Sigma_{2,h}^{(1)}
\times V_{1,h}$. Again, we observe the optimal convergence rates of
both stress and displacement.

\begin{table}[!ht] 
\small
\centering
%\begin{subtable}{
\begin{tabular}{p{.4cm}|p{1.3cm}p{.4cm}|p{1cm}p{.4cm}|p{1.1cm}p{.4cm}|p{1.2cm}p{.4cm}|p{.8cm}p{1cm}}
  \hline
  $1/h$	&$\|u -u_h\|_0$	&$h^n$	&$\|\bepsilon_h\|_0$&$h^n$
  &$\|\div_h\bepsilon_h\|_0$	&$h^n$ &
  $\|[\bsigma_h]\|_{0,\cF_h^i}$ & $h^n$	& $\dim V_{1,h}$ & $\dim
  \Sigma_{2,h}^{(1)}$  \\ \hline
  4		&0.01983 &--	&0.04152 &--	&0.57945	&--	& 0.02688 &-- & 192 	&
  416\\
  8		&0.00503 &1.98	&0.00821 &2.34	&0.14651	&1.98	& 0.00509 &
  2.40 & 768 	&
  1600\\
  16	&0.00126 &1.99	&0.00189 &2.12	&0.03674  &2.00	& 0.00092 &
  2.47
  & 3072	&6272	\\
  32	&0.00032 &2.00	&0.00046 &2.03  &0.00924	&1.99	& 0.00016 &2.49
  & 12288	&24832	\\
  \hline
\end{tabular}
\caption{The error, $\bepsilon_h = \bsigma - \bsigma_h$, and
  convergence order for 2D on uniform grids, $\Sigma_{2,h}^{(1)}$}
\label{tab:numerical-2D-2-k1}
%}
\end{table}

\subsection{3D Test}

The 3D pure displacement problem is computed on the unit cube $\Omega =
[0,1]^3$ with a homogeneous boundary condition that $u = 0$ on
$\partial \Omega$. Let the exact solution be 
\[
u = 
\begin{pmatrix}
2^4\\
2^5\\
2^6
\end{pmatrix}x(1-x)y(1-y)z(1-z).
\]
The true stress function $\bsigma$ and the load function $f$ can be
analytically derived from the \eqref{equ:elasticity} for a given solution $u$. 

\begin{table}[!ht] 
\small
\centering
%\begin{tabular}{c|cc|cc|cc|cc|cc}
\begin{tabular}{p{.4cm}|p{1.3cm}p{.4cm}|p{1cm}p{.4cm}|p{1.1cm}p{.4cm}|p{1.2cm}p{.4cm}|p{.8cm}p{1cm}}
  \hline
  $1/h$	&$\|u-u_h\|_0$	&$h^n$ &$\|\bepsilon_h\|_0$ &$h^n$
  &$\|\div_h\bepsilon_h\|_0$	&$h^n$	&
  $\|[\bsigma_h]\|_{0,\cF_h^i}$ & $h^n$	& $\dim V_{0,h}$ & $\dim
  \Sigma_{1,h}^{(1)}$  \\ \hline
  2		&0.22624 &--		&1.05758 &--		&8.05894	&--	&0.21689	& --	 &144 &936 \\ 
  4		&0.12549 &0.85	&0.47884 &1.14	&4.48971	&0.84 	&0.13908	&0.64	 &1152	&7200 \\
  8		&0.06345 &0.98	&0.20060 &1.25	&2.30280	&0.96	&0.05726	&1.28	 &9216	&56448\\
  16	&0.03175 &0.99	&0.09094 &1.14	&1.15867	&0.99	 &0.02104	&1.45	 &73728	&446976 \\
%  32	&xxxxx &1.00	&xxxxx &1.03  &xxxxx	&1.00	 &xxxxx	&xxxxx \\
  \hline
\end{tabular}
\caption{The error, $\bepsilon_h = \bsigma - \bsigma_h$, and
  convergence order for 3D on uniform grids, $\Sigma_{1,h}^{(1)}$}
\label{tab:numerical-3D-1}
\end{table}

\begin{table}[!ht] 
\small
\centering
%\begin{tabular}{c|cc|cc|cc|cc|cc}
\begin{tabular}{p{.4cm}|p{1.3cm}p{.4cm}|p{1cm}p{.4cm}|p{1.1cm}p{.4cm}|p{1.2cm}p{.4cm}|p{.8cm}p{1cm}}
  \hline
  $1/h$	&$\|u -u_h\|_0$	&$h^n$	&$\|\bepsilon_h\|_0$&$h^n$
  &$\|\div_h\bepsilon_h\|_0$	&$h^n$ &
  $\|[\bsigma_h]\|_{0,\cF_h^i}$ & $h^n$	& $\dim V_{0,h}$ & $\dim
  \Sigma_{1,h}^{(2)}$  \\ \hline
  2		&0.26120 &--		&1.39194 &--		&8.05894	&--	  & 0.28483 &--   &144		&378 \\ 
  4		&0.15504 &0.75	&0.78910 &0.81	&4.48917	&0.84	& 0.24513 &0.22 &1152		&2766 \\
  8		&0.07923 &0.97	&0.26868 &1.55	&2.30280	&0.96	& 0.12466 &0.98 &9216 	&21654 \\
  16	&0.03937 &1.01	&0.08303 &1.69	&1.15867  &0.99	& 0.04932 &1.34 &73728	&172326	\\
%%  32  &0.06297 &1.01  &0.15304 &xxxx  &0.58406  &0.99 & 0.02803 &1.34 &589824 & 1413702 \\
  \hline
\end{tabular}
\caption{The error, $\bepsilon_h = \bsigma - \bsigma_h$, and
  convergence order for 3D on uniform grids, $\Sigma_{1,h}^{(2)}$}
\label{tab:numerical-3D-2}
\end{table}

The numerical results when applying two classes of spaces on 3D uniform
grids are illustrated in Table \ref{tab:numerical-3D-1} and
\ref{tab:numerical-3D-2}. Here we set the parameter of the penalty term
as $\eta=1$ for the pair $\Sigma_{1,h}^{(1)}-V_h$, and  $\eta=0.1$ for the
pair $\Sigma_{1,h}^{(2)}-V_h$. It can be observed that, similar to the 2D
case, the optimal orders of convergence are achieved for two classes of
spaces. We also note that the global dimension of the space for stress
has been reduced by approximately $60\%$ for $\Sigma_{1,h}^{(2)}$.

\section{Concluding Remarks}
In this paper we propose mixed finite elements of any order for the
linear elasticity in any dimension. According to the stability for
$R_k^\perp$ and $R_k$, and the approximation property, we have the
following choices for the finite elements. 

\begin{table}[!ht] 
\small
\centering
\begin{tabular}{|c|c|c|c|}
  \hline
 & Stability of $R_k^\perp$ & Stability of $R_k$ & Approximation
 property \\ \hline 
$\Sigma_{k+1,h,b}$  &$\surd$&& \\ 
  \hline 
$\tilde{\Sigma}_{k+1,h,f}$ &&$\surd$& \\
  \hline 
$\tilde{\Sigma}_{k+1,h,b}$&&& if $\Sigma_{k+1,h,b}$ and
$\tilde{\Sigma}_{k+1,h,f}$ are chosen \\
  \hline 
$\Sigma_{k+1,h}^c$ &&&$\surd$ \\
  \hline 
$\Sigma_{k+1,h,f}^c$ &&$k\geq n$& \\
  \hline 
\end{tabular}
\caption{Different choices of spaces}
\label{tab:concluding}
\end{table}

Based on the Table \ref{tab:concluding}, we have three choices that
the three ingredients are satisfied: 
\begin{enumerate}
\item $\Sigma_{k+1,h,b} + \tilde{\Sigma}_{k+1,h,f} +
\tilde{\Sigma}_{k+1,h,b}$. We also prove that the sum is direct based
on the local decomposition of discrete symmetric tensors. The lower
order ($k\leq n-1$) finite element diagrams of this class are depicted
in Figure \ref{fig:diagram-2D} and \ref{fig:diagram-3D} for 2D and 3D,
respectively. 
\item $\Sigma_{k+1,h,b} + \tilde{\Sigma}_{k+1,h,f} +
\Sigma_{k+1,h}^c$. This class of finite elements does not have local
d.o.f. but has fewer global dimension.
\item $\Sigma_{k+1,h,b} + \Sigma_{k+1,h}^c$ for $k\geq n$. This
class of conforming elements has been found by Hu \cite{hu2015finite}. 
\end{enumerate}

\begin{figure}[!htbp]
\centering 
\captionsetup{justification=centering}
\subfloat[$k=0$]{\centering 
   \includegraphics[width=0.28\textwidth]{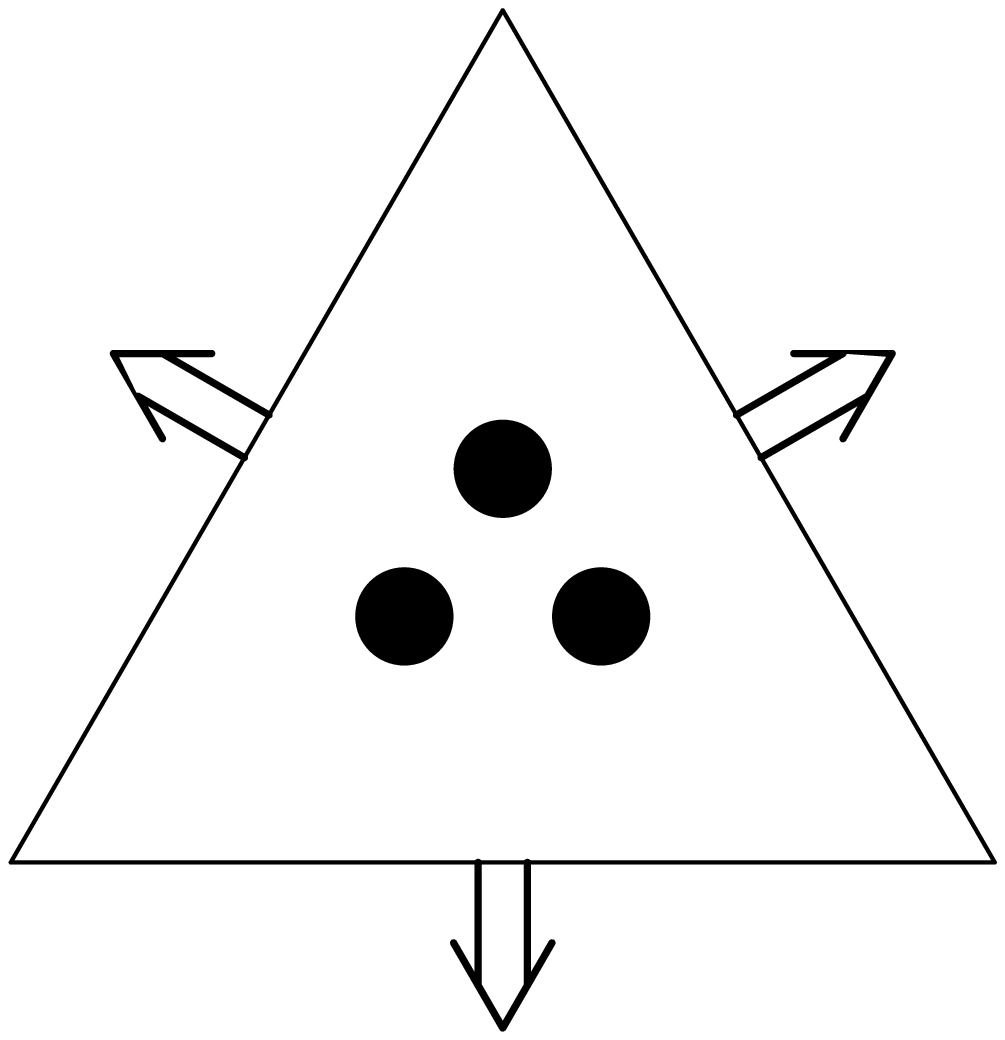} 
}% 
\quad 
\subfloat[$k=1$]{\centering 
   \includegraphics[width=0.28\textwidth]{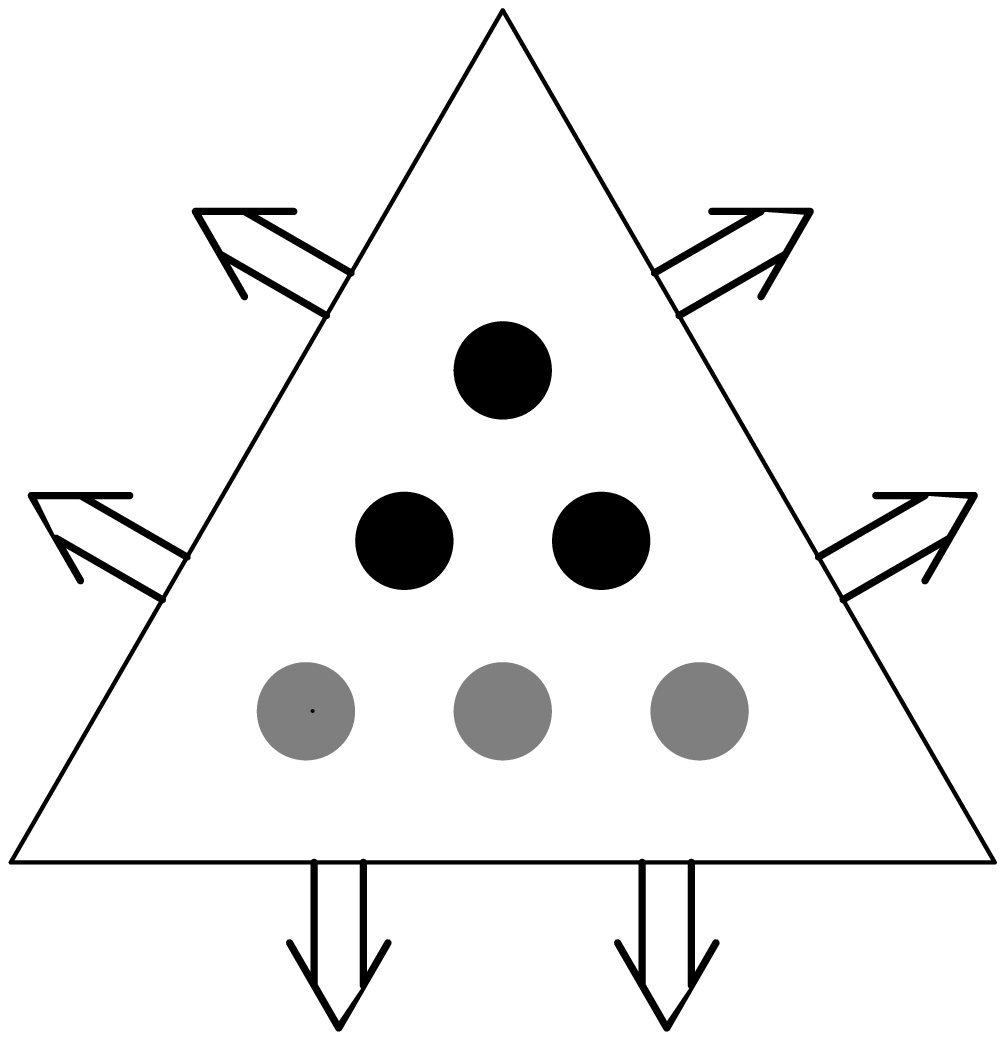} 
}
\caption{Element diagrams for $\tilde{\Sigma}_{k+1,h}^{(1)}$ in 2D \\ 
  gray circle: conforming div-bubble; 
  black circle: nonconforming div-bubble}
\label{fig:diagram-2D}
\end{figure}

\begin{figure}[!htbp]
\centering 
\captionsetup{justification=centering}
\subfloat[$k=0$]{
  \includegraphics[width=0.28\textwidth]{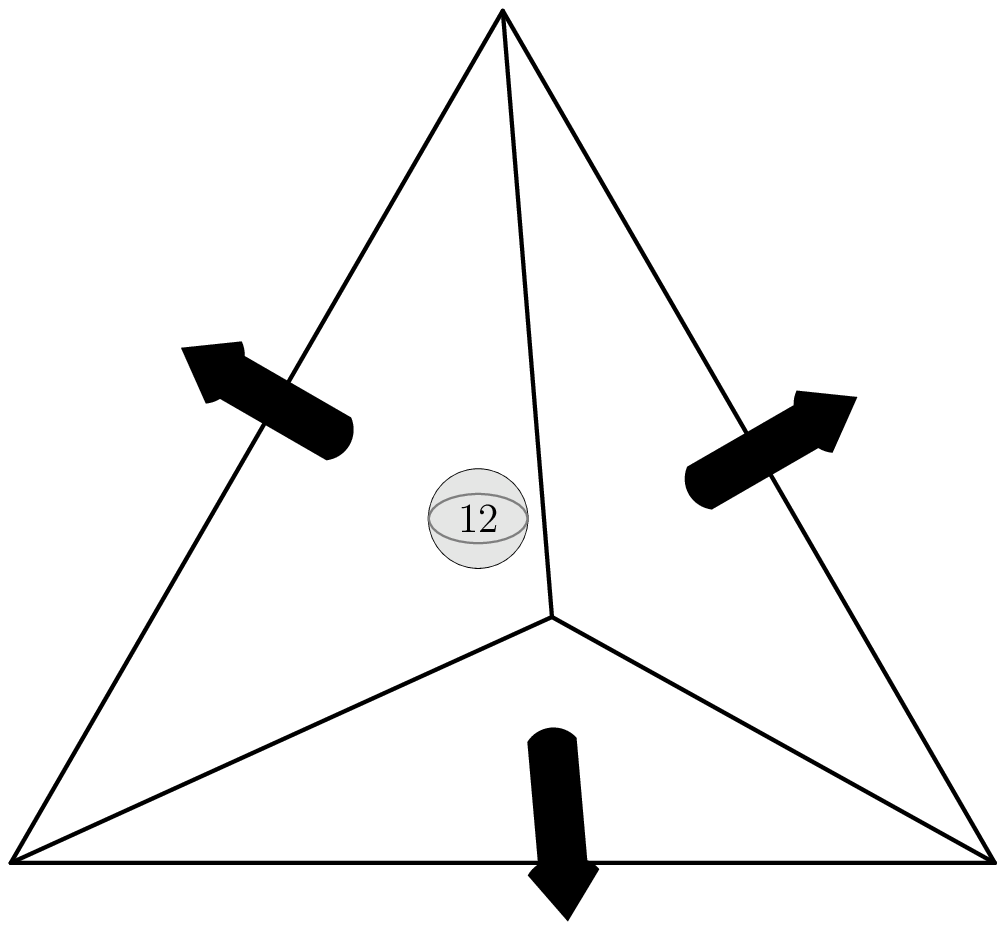} 
}%
\subfloat[$k=1$]{
  \includegraphics[width=0.28\textwidth]{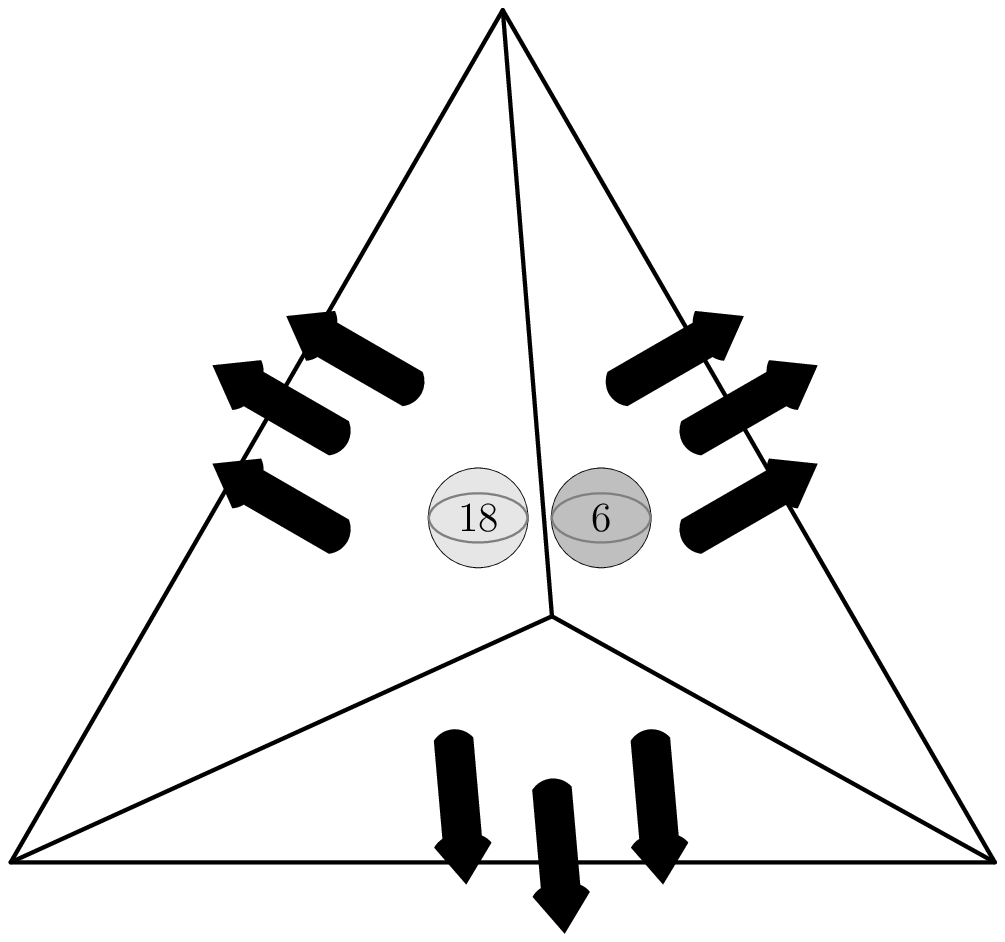} 
}%
\subfloat[$k=2$]{
  \includegraphics[width=0.28\textwidth]{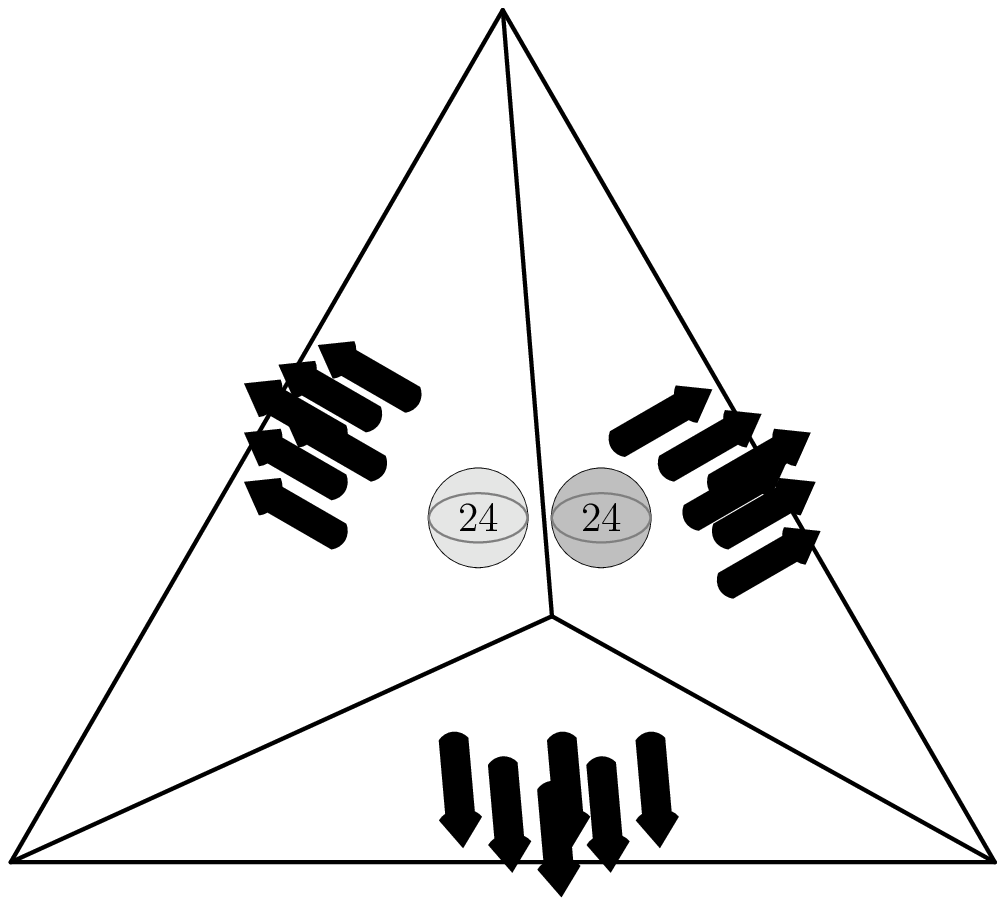}
}
\caption{Element diagrams for $\tilde{\Sigma}_{k+1,h}^{(1)}$ in 3D \\
  dark gray ball: conforming div-bubble; light gray ball:
  nonconforming div-bubble
}
\label{fig:diagram-3D}
\end{figure}

For consistency, an interior penalty term is added to the
bilinear form, which will improve the convergence order but not affect
the stability. One main advantage of these finite elements is their
convenience for implementation, since the basis functions of
nonconforming face-bubble function spaces can be written explicitly in
terms of the orthonormal polynomials. For the case that $k \leq n-1$,
we prove that the nonconforming elements have to be applied in the
framework that the degree of polynomials for stress are at most $k+1$.

\paragraph{Acknowledgement} The author would like to thank Professor
Jun Hu for the helpful discussions and suggestions.

%% end of file %%

\bibliographystyle{siam}
\bibliography{2016Nonconforming.bib}

\end{document}